\documentclass[a4paper,oneside,8pt]{article}
\usepackage{amsmath}
\usepackage{amsthm,amssymb,enumerate,hyperref}
\usepackage[a4paper,margin=2cm,centering,nohead,ignorefoot,footskip=5ex]{geometry}
\usepackage{url}
\usepackage{xcolor}
\hypersetup{
    colorlinks,
    linkcolor={red!50!black},
    citecolor={red!50!black},
    urlcolor={blue!80!black}
}
\theoremstyle{plain}
\newtheorem{theorem}{Theorem}
\newtheorem{proposition}[theorem]{Proposition}
\newtheorem{notation}[theorem]{Notation}
\newtheorem{lemma}[theorem]{Lemma}
\theoremstyle{definition}
\newtheorem{remark}[theorem]{Remark}
\newtheorem{definition}[theorem]{Definition}
\newtheorem{corollary}[theorem]{Corollary}
\newcommand{\N}{\mathbb{N}}
\newcommand{\R}{\mathbb{R}}
\newcommand*{\mforall}{\tilde\forall}
\newcommand*{\mexists}{\tilde\exists}
\newcommand{\norm}[1]{\left\lVert#1\right\rVert}
\newcommand{\BFI}{\mathsf{BFI}}

\newcommand{\mPPA}{\mathsf{mPPA}}

\begin{document}

\title{Metastability of the proximal point algorithm with multi-parameters
}
\author{Bruno Dinis\thanks{Departamento de Matem\'atica, Faculdade de Ci\^encias da
Universidade de Lisboa, Campo Grande, Edif\'icio~C6, 1749-016~Lisboa, Portugal. \protect\url{bmdinis@fc.ul.pt}.}
\and
Pedro Pinto\thanks{Department of Mathematics, Technische Universit{\"a}t Darmstadt, Schlossgartenstrasse 7, 64289 Darmstadt, Germany. \newline \protect\url{pinto@mathematik.tu-darmstadt.de}.}}

\maketitle

\begin{abstract}
In this article we use techniques of proof mining to analyse a result, due to Yonghong Yao and Muhammad Aslam Noor, concerning the strong convergence of a generalized proximal point algorithm which involves multiple parameters. Yao and Noor's result ensures the strong convergence of the algorithm to the nearest projection point onto the set of zeros of the operator. Our quantitative analysis, guided by Fernando Ferreira and Paulo Oliva's bounded functional interpretation, provides a primitive recursive bound on the metastability for the convergence of the algorithm, in the sense of Terence Tao. Furthermore, we obtain quantitative information on the asymptotic regularity of the iteration. The results of this paper are made possible by an arithmetization of the $\limsup$.
\end{abstract}

\section{Introduction}

In the theory of maximal monotone operators, many problems, such as problems of minimization of a function, variational inequalities, etc., can be formulated as finding a zero of a maximal monotone operator (see e.g.\ \cite{AK(98)} and the references therein). The proximal point algorithm $\eqref{RockPPA}$ \cite{Rockafellar76} is a powerful and successful algorithm in finding a solution of maximal monotone operators. Starting from any initial guess $x_0 \in H$, the $\eqref{RockPPA}$ generates a sequence $(x_n)$ which approximates the solution, defined by
\begin{equation}\tag{$\mathsf{PPA}$}\label{RockPPA}
x_{n+1}=\mathsf{J}_{c_n}(x_n)+e_n,
\end{equation}
where $\mathsf{J}_{c_n}$ are resolvent functions of a maximal monotone operator with parameter $c_n>0 $, and $(e_n)$ an error sequence. Ralph Rockafellar showed that  \eqref{RockPPA} converges weakly towards a zero of the operator, provided that $(c_n)$ is bounded away from zero and $\norm{e_n}$ is summable, i.e.\ $\sum_{n=0}^{\infty}\norm{e_n}<\infty$. Osman G\"{u}ler showed in \cite{G(91)}, by providing a counter-example, that in general \eqref{RockPPA} does not converge strongly. For this reason, several modifications of the algorithm were studied (see for example \cite{EB(92),HH(01),MX(04)}).

The Krasnosel'ski\u{\i}-Mann ($\mathsf{KM}$) iteration is an important algorithm for the approximation of fixed points of nonexpansive maps \cite{Mann}. One relevant feature that makes the $\mathsf{KM}$ iteration attractive is the fact that it is Fejér monotone relative to the set of fixed points \cite{Combettes(09), KLN(18)}. In general, one can only guarantee weak convergence for the $\mathsf{KM}$ iteration. Of course, if in practical optimization problems we restrict ourselves to a finite-dimensional context, then this limitation disappears. On the other hand, with the so-called Halpern variant iteration \cite{Halpern67} one can actually establish a general strong convergence result. The fact that the Halpern iteration seems to be less sought after for optimization practice may rest on the fact that this iteration is no longer Fejér monotone. However, recent results brought a renewed interest to the Halpern type iteration (see e.g. \cite{Lieder(17),Diak(20)}). Through speedup techniques (as in \cite{Nesterov(07)}), improvements on the rate of asymptotic regularity were obtained for the Halpern iteration. 

The impact of these iterations in fixed point theory motivated the following variants of \eqref{RockPPA}:

\begin{equation}\tag{$\mathsf{KM\text{-}PPA}$}\label{KM-PPA}
x_{n+1}=\lambda_n x_n+(1-\lambda_n)\mathsf{J}_{c_n}(x_n)+e_n,
\end{equation}

\begin{equation}\tag{$\mathsf{H\text{-}PPA}$}\label{H-PPA}
x_{n+1}=\lambda_nu+(1-\lambda_n)\mathsf{J}_{c_n}(x_n)+e_n,
\end{equation}
with $(\lambda_n) \subset [0,1]$, and for all $n \in \N$, $c_n>0 $ and $e_n$ is an error term. While, in general, \eqref{KM-PPA} is only weakly convergent \cite{KT(00)}, the algorithm \eqref{H-PPA} is strongly convergent  \cite{KT(00),X(02)}. Other strongly convergent variants of the proximal point algorithm are, for example, the hybrid projection $\mathsf{PPA}$ \cite{SS(00)}, the shrinking projection $\mathsf{PPA}$ \cite{CAY(10)} and the viscosity $\mathsf{PPA}$ \cite{Taka(07)}. 

In this paper, we will focus on the following multi-parameter version of the proximal point algorithm, considered by  Yonghong Yao and Muhammad Aslam Noor in \cite{YaoNoor2008}, which is an hybrid between \eqref{KM-PPA} and \eqref{H-PPA}. Let $z_0 \in H$ be an initial guess and define

\begin{equation}\tag{$\mPPA$}\label{PPA}
z_{n+1}=\lambda_nu+\gamma_nz_n+\delta_n\mathsf{J}_{c_n}(z_n)+e_n,
\end{equation}
where $u \in H$ is given, and for all $n \in \N$ it holds that $c_n>0 $, $\lambda_n,\gamma_n,\delta_n \in (0,1)$ and $\lambda_n+\gamma_n+\delta_n=1$.

Yao and Noor showed in \cite{YaoNoor2008} that the algorithm \eqref{PPA} is strongly convergent to the nearest projection point onto the set of zeros of the operator (see Theorem~\ref{ThmYaoNoor}). The goal of this paper is to study quantitative information regarding the strong convergence of the iteration \eqref{PPA}.

Using proof-theoretical techniques, we analyse Yao and Noor's proof, and are able to prove the strong convergence of \eqref{PPA} using only a weaker version of the metric projection principle, an arithmetization of the $\limsup$, and bypassing the use of a sequential weak compactness argument. Theorem~\ref{theoremyaonoor2} provides an effective bound on the metastability -- in the sense of Terence Tao \cite{T(08a),T(08b)} -- of \eqref{PPA}, i.e.\ we obtain a computable function $\phi: \N \times \N^{\N}\to \N $ such that \begin{equation}\label{star}\tag{$\star$}
\forall k \in \N \,\forall f : \N \to \N \,\exists n \leq \phi(k,f) \, \forall i,j \in [n,n+f(n)] \left(\norm{z_i-z_j}\leq \frac{1}{k+1} \right).
\end{equation}
Note that  \eqref{star} is equivalent (in a non-effective way) to the Cauchy property of the sequence $(z_n)$. For a general sequence $(z_n)$ this is the best one can hope for, in the sense that it is not possible to obtain a computable rate for the Cauchy property. Furthermore, we obtain quantitative information on the asymptotic regularity of the iteration (Proposition~\ref{lemmaintermediate} and Corollary~\ref{Corasympreg}). The complexity of the extracted information follows directly from the strength of the logical principles required for the proof. The fact that these results are proved using the weaker principles mentioned above is reflected in the bounds obtained, which are primitive recursive (in the sense of Kleene). The way to deal with the projection and sequential weak compactness arguments was already understood \cite{kohlenbach2011quantitative,FFLLPP(19)}, so the remaining obstacle was on how to deal with the $\limsup$. The problem is solved by an arithmetization of the $\limsup$ (Lemma~\ref{lemmarationalapprox2}) tailored to deal with a key argument (Lemma~\ref{SuzukiLemma2}) in Yao and Noor's proof (see also \cite{KS(ta)} for an alternative way to deal with the $\limsup$).

The methods used in this paper are set in the framework of proof mining, a program that describes the process of using proof-theoretical techniques to analyse mathematical proofs with the aim of extracting new information. This program was developed mainly by the works of Ulrich Kohlenbach and his collaborators, producing a vast number of results; analysing proofs from areas such as approximation theory, ergodic theory, fixed point theory and optimization theory (see e.g.\ \cite{K(17),KO(03),K(08)}). In the proof mining program, proof interpretations are used as tools to extract constructive (i.e.\ computational) information from noneffective proofs. The output of the interpretation, which we refer to as \emph{quantitative version}, gives explicit information that previously was implicit and hidden behind the use of quantifiers. The standard technique in proof mining is Kohlenbach's \emph{monotone functional interpretation} \cite{K(96),K(08)} which is based on Kurt G\"{o}del's \emph{Dialectica} interpretation \cite{G(58),AF(98)} that works with upper bounds for witnessing terms instead of precise witnesses. Recently, Fernando Ferreira and Paulo Oliva's \emph{bounded functional interpretation} ($\BFI$) \cite{FO(05)} (see also the classical version \cite{F(09)}) has proven to be a valid alternative for the proof mining program, providing new insight to some theoretical questions concerning the elimination of sequential weak compactness \cite{FFLLPP(19)}.

  The quantitative results, as well as their proofs, obtained by the proof mining program do not presuppose any particular knowledge of logical tools because the latter are only used as an intermediate step and are not visible in the final product. Apart from the end of Section~\ref{Sectionlogic}, where we make some remarks on the logical aspects of the analysis, our work is no exception and so, knowledge of tools from logic in general and familiarity with the $\BFI$ in particular are not necessary to read this paper. 

We would like to point out that our work comes as a natural generalization of \cite{LNS(18),LS(18),LLPP(ta)}.

The structure of the paper is the following.
In Section~\ref{SectionStrong} we recall some notions concerning Hilbert spaces and monotone operators. We also state the result by Yao and Noor for which we give a quantitative analysis in the subsequent sections and explain its original proof.
In Section~\ref{SectionRestricting} we give a quantitative treatment of the principles used in the original proof: the projection argument, sequential weak compactness, and the arithmetization of the $\limsup$.
The quantitative analysis of Yao and Noor's proof is carried out in Section~\ref{SectionQuantitative}. We start by obtaining a partial result which depends on an additional condition (Subsection~\ref{SectionMetastability}). This condition is then shown to be satisfied by a concrete functional (Subsection~\ref{SectionMainLemmas}). Finally, we prove the main result regarding the metastability for \eqref{PPA} (Subsection~\ref{Sectionrevisited}). 
We leave some final remarks that allow to better understand some theoretical aspects of the analysis to Section~\ref{Sectionlogic}.

\section{Preliminaries}\label{SectionStrong}

We work in a real Hilbert space $H$ with inner product $\langle \cdot, \cdot \rangle$ and norm $\norm{\cdot}$. We recall that a multi-valued operator $T:H \to 2^{H}$  is said to be \emph{monotone} if and only if whenever $(x,y)$ and $(x',y')$ are elements of the graph of $T$, it holds that $\langle x-x',y-y'\rangle \geq 0$. A monotone operator $T$ is \emph{maximal monotone} if  the graph of $T$ is not properly contained in the graph of any other monotone operator on $H$. We denote by $S$ the set of all \emph{zeros} of $T$, i.e.\ $S = T^{-1}(0)$. 
We fix $T$ a maximal monotone operator and assume henceforth $S$ to be nonempty.
For $c>0$, we use $\mathsf{J}_c$  to denote the \emph{resolvent function} of $T$ with parameter $c$, i.e.\ the single-valued function defined by
\begin{equation*}
\mathsf{J}_c = (I + cT )^{-1}.
\end{equation*}
A mapping $f : H \to H$ is called \emph{nonexpansive} if
$$\forall x, y \in H (\norm{f (x) -f (y)}\leq \norm{x -y}).$$ 
The set of fixed points of a mapping $f$ is the set  $\mathrm{Fix}(f):=\{x \in H: f(x)=x\}$.  The resolvent mapping $\mathsf{J}_c$ is nonexpansive, and for every $c>0$, the set of fixed points of $\mathsf{J}_c$ is $S$.
If $f$ is nonexpansive, then $\mathrm{Fix}(f)$ is a closed and convex subset of $H$. For a comprehensive introduction to convex analysis and the theory of monotone operators in Hilbert spaces we refer to \cite{BC(17)}.

The following lemmas are well-known.

\begin{lemma}[Resolvent identity]
For every $a,b>0$ and every $x \in H$ it holds that 
\begin{equation*}
\mathsf{J}_a (x) = \mathsf{J}_b\left(\frac{b}{a}x+ \left(1-\frac{b}{a} \right)\mathsf{J}_a (x)\right).
\end{equation*}
\end{lemma}

\begin{lemma}[\cite{MX(04)}]\label{lemmaresolvineq}
If $0<a\leq b$, then for all $x \in H$ it holds that $\norm{\mathsf{J}_a( x)-x} \leq 2 \norm{\mathsf{J}_b (x)-x}$.
\end{lemma}

The quantitative analysis carried out in this paper makes use of the notion of monotone functional. For that matter we consider the notion of strong majorizability $\leq^{\ast}$ from \cite{bezem1985strongly} in the following two particular cases. In the case of functions $f,g:\N\to\N$, we have
$$g\leq^* f := \forall n\, \forall m\leq n\, \left( g(m)\leq f(n) \land f(m)\leq f(n)\right),$$
and in the case of functionals $\varphi, \psi:\N\times \N^{\N}\to\N$, we have
$$\varphi\leq^*\psi := \forall n \,\forall f: \N \to \N \,\forall m\leq n \,\forall g\leq^* f\, \left(\varphi(m,g)\leq \psi(n,f) \land \psi(m,g)\leq \psi(n,f)\right).$$
We say that $f$ is monotone if $f\leq^*f$ and similarly for a functional $\varphi$. For functions $f:\N\to\N$, this monotone property coincides with the usual property $\forall n\in \N\, \left(f(n)\leq f(n+1)\right)$. Quantifications over monotone functions $\forall f (f\leq^* f \to \cdots)$ will be abbreviated by $\mforall f (\cdots)$. Due to the particularities of the $\BFI$, our quantitative results quantify over such monotone functions. Note however that for all $f: \N \to \N$, one has $f \leq^* \! f^{\mathrm{maj}}$, where $f^{\mathrm{maj}}(n):= \max_{i \leq n}\{f(i)\}$. Hence there is no real restriction in working with monotone quantifications.

 \subsection{A result by Yao and Noor}\label{Sectionoriginalproof} 
We present the result by Yao and Noor concerning the strong convergence of the $\eqref{PPA}$ (Theorem~\ref{ThmYaoNoor}) and give a detailed description of its proof. 
 This will hopefully guide the reader through our quantitative analysis and the several steps that it requires. 
 
Consider the following set of conditions.
\begin{enumerate}[($H_1$)]
\item\label{H1} $\lim_{n \to \infty} \lambda_n=0$.
\item\label{H2} $\sum_{n=0}^{\infty} \lambda_n=\infty$.
\item\label{H3} $0<\liminf_{n \to \infty}\gamma_n\leq \limsup_{n \to \infty}\gamma_n<1$.
\item\label{H4} $c_n \geq c,$ where $c$ is some positive constant.
\item\label{H5} $c_{n+1}-c_n \to 0$.
\item\label{H6} $\sum_{n=1}^{\infty}\norm{e_n}<\infty$.
\end{enumerate}


\begin{theorem}{\rm (\cite[Theorem 3.3]{YaoNoor2008})}\label{ThmYaoNoor}
Let $(z_n)$ be generated by \eqref{PPA}.
Assume that $(H_{\ref{H1}})$-$(H_{\ref{H6}})$ hold.
Then $(z_n)$ converges strongly to a point $z \in S$, the nearest to $u$.
\end{theorem}


The proof of Theorem~\ref{ThmYaoNoor} relies on Lemma~\ref{SuzukiLemma1} and Lemma~\ref{SuzukiLemma2} below, due to Tomonari Suzuki \cite{Suzuki2005}, for which we give quantitative versions in Subsection~\ref{SectionMainLemmas} (Lemma~\ref{lemmaSuzuki1} and Lemma~\ref{LemmaSuzuki2}).\footnote{In fact, Lemma \ref{SuzukiLemma1} is only used to prove Lemma~\ref{SuzukiLemma2}.} 

\begin{lemma}\textup{(\cite[Lemma 2.1]{Suzuki2005})}\label{SuzukiLemma1}
Let $(z_n)$ and $(w_n)$ be sequences in a Banach space $X$ and let $(\alpha_n)$ be a sequence in $[0,1]$ such that $\limsup \alpha_n < 1$. Suppose that $z_{n+1} = \alpha_nw_n +(1-\alpha_n)z_n$ for all $n \in \N$,
$\limsup (\norm{w_{n+1} -w_n}-\norm{z_{n+1}-z_n})\leq 0$ and $d := \limsup \norm{w_n -z_n}<\infty$. Then
\begin{equation*}
\forall t \in \N \left(\liminf \lvert \,\norm{w_{n+t}-z_n}-(1+\alpha_n+\dots+\alpha_{n+t-1})d \,\rvert =0\right).
\end{equation*}
\end{lemma}

\begin{lemma}\label{SuzukiLemma2}\textup{(\cite[Lemma 2.2 ]{Suzuki2005})}
Let $(z_n)$ and $(w_n)$ be bounded sequences in a Banach space $X$ and let $(\alpha_n)$ be a sequence in $[0,1]$ with $0< \liminf \alpha_n \leq \limsup \alpha_n <1$. Suppose that $z_{n+1}=\alpha_{n}w_n+(1-\alpha_n)z_n$ for all $n \in \N$, and $\limsup (\norm{w_{n+1} -w_n}-\norm{z_{n+1}-z_n})\leq 0$. Then $\lim \norm{w_n-z_n}=0$.
\end{lemma}

The proof of Theorem~\ref{ThmYaoNoor} is divided in the following steps:
\begin{enumerate}[(1)]
\item \underline{Show that $(z_n)$ is bounded.} This is just a simple proof by induction and some easy computations.
\item \underline{$\lim \norm{z_{n+1}-z_n}=0$.} Letting $z_{n+1}=\gamma_nz_n+(1-\gamma_n)w_n$, it is shown first that $\limsup (\norm{w_{n+1}-w_n}-\norm{z_{n+1}-z_n})\leq 0$. Using Lemma~\ref{SuzukiLemma2} one concludes that $\lim \norm{z_n-w_n}=0$ which,  by the definition of $w_n$, is enough to conclude this step.
\item \underline{$\lim \norm{\mathsf{J}_{c}(z_n)-z_n}=0$.} From step (2) and the hypothesis of the theorem one concludes that $\lim \norm{\mathsf{J}_{c_n}(z_n)-z_n}=0$. The conclusion follows from Lemma~\ref{lemmaresolvineq}.
\item \underline{Projection argument.} With $\tilde{p}$ the projection point of $u$ onto $S$ it is shown that $\forall q \in S (\langle \tilde{p}-u,\tilde{p}-q\rangle\leq 0)$. 
\item \underline{Sequential weak compactness and demiclosedness.} Pick a subsequence $(z_{n_j})$ of $(z_n)$ such that $\limsup \langle \tilde{p}-u,\tilde{p}-z_n\rangle= \lim_{j}\langle \tilde{p}-u,\tilde{p}-z_{n_j}\rangle $ and $(z_{n_j})$ converges weakly to some $q\in S$. Here the following demiclosedness principle is used.

\begin{lemma}[Demiclosedness principle \cite{B(65)}]
Let $C$ be a closed convex subset of $H$ and let $f : C \to C$ be a nonexpansive mapping such that $\mathrm{Fix}(f)\neq \emptyset$. Assume that $(x_n)$ is a sequence in $C$ such that $(x_n)$ weakly converges to $x \in C$ and  $((I - f)(x_n))$ converges strongly to $y \in H$. Then $(I - f)(x) = y$.
\end{lemma}

By step (4) it follows that $\limsup \langle \tilde{p}-u,\tilde{p}-z_n\rangle \leq 0$.
\item  \underline{Main combinatorial part.} In this final step it is shown that the conditions of Lemma~\ref{LemmaXu} below are satisfied. The application of this lemma is enough to conclude the result.

\begin{lemma}\textup{(\cite[Lemma~2.5]{X(02)})}\label{LemmaXu}
Let $(a_n)$ be a sequence of nonnegative real numbers such that for all $n \in \N$
\begin{equation}\label{xu_lem_main_ass}
a_{n+1}\leq (1 - s_n)a_n + s_nt_n + \delta_n,
\end{equation}
where $(s_n)\subset  [0, 1]$ and $(t_n), (\delta_n)$ are such that $\sum_{n=0}^{\infty}s_n=\infty$, $\limsup_{n \to \infty}t_n\leq 0$ and $\sum_{n=0}^{\infty} \delta_n <\infty$. Then $(a_n)$ converges to zero.
\end{lemma}
\end{enumerate}

\section{Avoiding roadblocks}\label{SectionRestricting}

From the point of view of proof mining the analysis of the original proof of Theorem~\ref{ThmYaoNoor} presents some difficulties that prevent the extraction of simple bounds. These concern the projection argument in step (4), the weak compactness in step (5) and the assumed existence of the $\limsup$ in Suzuki's lemmas. In Subsections~\ref{SectionProjection} and \ref{Sectionswc} we adapt the way to deal with projection \cite{kohlenbach2011quantitative} and sequential weak compactness \cite{FFLLPP(19)} to our context. Furthermore, in Subsection~\ref{Sectionswc} we give a quantitative version of Lemma~\ref{LemmaXu}. In Subsection~\ref{SectionBypassACA} we give an arithmetization of the $\limsup$ that allows to obtain a quantitative version of Lemma~\ref{SuzukiLemma2}. A more detailed explanation to why these principles are problematic will be given in Section~\ref{Sectionlogic}.
 
\subsection{Projection argument}\label{SectionProjection}

In this section we deal with the following projection argument which is used in the original proof
\begin{equation}\label{projarg}
 \exists p\in S \, \forall k\in \N \,\forall q\in S  \left(\|p-u\|^2\leq \|q-u\|^2 + \frac{1}{k+1}\right),
 \end{equation}
 stating that there is a zero of $T$ that is the nearest to a given $u \in H$. 
The squares are added here only for an easier connection to the inner product of the space that will be required below.

As noticed by Kohlenbach \cite{kohlenbach2011quantitative}, instead of \eqref{projarg}, it is enough for the quantitative analysis to consider the weaker statement
\begin{equation}\label{projarg_weak}
\forall k\in \N \,\exists p\in S\, \forall q\in S  \left(\|p-u\|^2\leq \|q-u\|^2 + \frac{1}{k+1}\right).
\end{equation}

 While the proof of \eqref{projarg} requires the use of countable choice, the statement \eqref{projarg_weak} can be proved by a simple induction argument and this fact is reflected on the extracted bounds, which are then recursively defined. 

An analysis of the projection argument via the monotone functional interpretation was previously carried out by Kohlenbach \cite{kohlenbach2011quantitative}. A detailed explanation of the analysis of the projection argument using the bounded functional interpretation was shown in \cite{FFLLPP(19)}. In the latter, the assumption that one is working in a bounded set plays an important role in simplifying the interpretation. Although we do not have that assumption here, we can equivalently consider the projection statement in \eqref{projarg} restricted to a big enough ball centered at some zero point $s \in S$. This restriction was considered in \cite{PP(ta)}, giving rise to the quantitative version in Proposition~\ref{projection} bellow.

\begin{notation}
From now on, we will write $\mathsf{J}$ instead of $\mathsf{J}_{\frac{1}{c}}$ and $\mathsf{J}_n$ instead of $\mathsf{J}_{c_{n}}$, under the assumption that $c\in \N \setminus \{0 \}$ satisfies $c_n \geq \frac{1}{c}$ for all $n \in \N$.
For each $r\in \N$ and function $f:\N \to \N$ we denote the $r$-fold iteration of $f$ by $f^{(r)}$. I.e.\ $f^{(0)}\equiv {\rm Id}$ and $f^{(r+1)}=f(f^{(r)})$.
Let $s$ be a zero of $T$. For any $N\in \N$, let ${\rm B}_{N}:=\left\{x\in H : \|x-s\|\leq N\,\right\}$ denote the closed ball centered at $s$ with radius $N$. The zero point is always made clear by the context.
\end{notation}

\begin{proposition}[\cite{PP(ta)}]\label{projection}
	Let $N\in \N\setminus \{0\}$ be such that $N\geq 2\|u-s\|$ for some $s\in S$.\\
	For any natural number $k$ and monotone function $f:\N \to \N $, there are $n \leq f^{(r)}(0)$ and $p\in {\rm B}_{N}$ such that
	$$\|\mathsf{J}(p)-p\| \leq \frac{1}{f(n)+1}$$
	and
	$$\forall q \in {\rm B}_{N}\, 
	\left(\|\mathsf{J}(q)-q \|\leq \frac{1}{n+1} \to \|p-u\|^2 \leq \|q-u\|^2 + \frac{1}{k+1}\right),$$
	where $r:=N^2(k+1)$.
\end{proposition}

In Proposition~\ref{projection}, we considered only the (almost) fixed points of $\mathsf{J}$ since the fixed point set of all resolvent functions coincide. We prove the following quantitative version which requires majorizing information on the sequence of real numbers $(c_n)$.  
\begin{lemma}\label{samefix-pt-set}
	Let $(c_n) \subset \R^+$ and $c\in \N\setminus\{0\}$ satisfying $(Q_{\ref{ineqcn}})$. Consider $\mathcal{C}:\N\to\N$ a monotone function satisfying $\forall n\in\N\, \left( c_n\leq \mathcal{C}(n) \right)$. Let $\zeta:\N\times \N\to\N$ be the monotone function defined by
	$\zeta(k,n):=\mathcal{C}(n)c(k+1)-1$. For any $k,n\in \N$ and any $p\in H$,
	\begin{equation}\label{zeta_monotone}
	\norm{\mathsf{J}(p)-p}\leq \frac{1}{\zeta(k,n)+1} \;\to\; \forall n'\leq n \left(\norm{\mathsf{J}_{n'}(p)-p}\leq \frac{1}{k+1}\right).
	\end{equation}
\end{lemma}

\begin{proof}
	Let $e=\mathsf{J}(p)-p$ and assume $\|e\|\leq \dfrac{1}{\zeta(k,n)+1}$.	For $n'\leq n$, we have
	\begin{equation*}
	\begin{split}
	\mathsf{J}(p)=p+e &\leftrightarrow p\in p+e+\frac{1}{c}T(p+e) \\ &\leftrightarrow -e\cdot c\in T(p+e)\\
	&\leftrightarrow p+(1-c\cdot c_{n'})e \in p+e+c_{n'}T(p+e) \\& \leftrightarrow \mathsf{J}_{n'}(p+(1-c\cdot c_{n'})e)=p+e.
	\end{split}
	\end{equation*}
	Hence
	\begin{align*}
	\norm{\mathsf{J}_{n'}(p)-p}&\leq \norm{\mathsf{J}_{n'}(p)-\mathsf{J}_{n'}(p+(1-c \cdot c_{n'})e)}+\norm{\mathsf{J}_{n'}(p+(1-c \cdot c_{n'})e)-p}\\
	&\leq (1+|1-c\cdot c_{n'}|)\|e\|=c \cdot c_{n'}\|e\|\leq \frac{1}{k+1},
	\end{align*}

since $c_{n'} \leq \mathcal{C}(n)$.
\end{proof}

\subsection{Sequential weak compactness}\label{Sectionswc}

In the proof by Yao and Noor, sequential weak compactness, together with the demiclosedness principle, is used to show that $\limsup \, \langle \tilde{p}-u,\tilde{p}-z_m\rangle \leq 0$, where $\tilde{p}$ is the projection point of $u$ onto $S$ and $(z_n)$ is generated by \eqref{PPA}. This means that there exists $\tilde{p}\in S$  (the projection point) such that
\begin{equation}\label{weakcompactnessorig}
\forall k \in\N \,\exists n \in \N \,\forall m \geq n \left(\langle \tilde{p}-u,\tilde{p}-z_m\rangle \leq \frac{1}{k+1}\right).
\end{equation}
Without having access to the projection point, and working only with approximations, we switch from \eqref{weakcompactnessorig} to the following weakening 
\begin{equation*}
\forall k \in\N \,\exists p \in S \,\exists n \in \N \,\forall m \geq n \left(\langle p-u,p-z_m\rangle \leq \frac{1}{k+1}\right).
\end{equation*}
As it turns out, this weaker form is still enough to carry out the proof, avoiding the use of sequential weak compactness. This is similar to the arguments in the beginning of section~2 in \cite{FFLLPP(19)}.

\begin{proposition}[\cite{PP(ta)}]\label{projection3}
	 Let  $N\in \N\setminus \{0\}$ be a natural number satisfying $N\geq 2\|u-s\|$ for some point $s\in S$.
	 For any $k\in \N$ and monotone function $f:\N \to \N $, there are $n \leq 24N(\check{f}^{(R)}(0)+1)^2$ and $p\in {\rm B}_N$ such that
	 $$\|J(p)-p \| \leq \frac{1}{f(n)+1} \, \land \, \forall q\in B_N 
	 \left(\|J(q)-q\|\leq \frac{1}{n+1} \to \langle u-p,q-p\rangle \leq \frac{1}{k+1}\right),$$ 
	 with $R:=4N^4(k+1)^2$ and $\check{f}:=\max\{ f(24N(m+1)^2),\, 24N(m+1)^2 \}$.	
\end{proposition}

In Proposition~\ref{lemmaintermediate}\eqref{limzn-Jc} below (under an additional assumption \eqref{hypchi} on a functional $\chi$) we compute a monotone functional $\xi_{\chi}$ satisfying
\begin{equation}\label{future_xi}
\forall k \in \N \,\mforall f:\N \to \N \,\exists n \leq \xi_{\chi}(k,f)\,\forall m \in [n,n+f(n)]\left(\norm{\mathsf{J}(z_m)-z_m}\leq \frac{1}{k+1} \right).
\end{equation}

Furthermore, in Lemma \ref{lemmamain1}, we show that the sequence $(z_n)$ is bounded and compute an explicit natural number $N_0$ such that $(z_n)\subset {\rm B}_{N_0}$. 

Proposition~\ref{removal_weakcompact} corresponds to the elimination of the sequential weak compactness argument. It can be seen as an application of the general principle in \cite[Proposition 4.3]{FFLLPP(19)} with $\alpha(k,f)=\xi_{\chi}(k,f+1)$, where the sequence being considered is $(z_{m+1})$, $\beta$ is given by Proposition \ref{projection3} and $\varphi(x,y)=\langle x-u, y\rangle$.
\begin{proposition}\label{removal_weakcompact}
		Let $\xi_{\chi}: \N\times \N^{\N} \to \N$ be a monotone function satisfying \eqref{future_xi}. For some $s \in S$, let $N_0 \in \N$ be  such that $(z_n)\subset {\rm B}_{N_0}$, and $N \geq \max\{2 \norm{u-s},N_0\}$. 
		For any $k\in \N$ and monotone function $f:\N \to \N$, there are $n\leq \psi_{\chi}(k,f)$ and $p\in {\rm B}_{N}$ such that $ \norm{\mathsf{J}(p)-p}\leq \frac{1}{f(n)+1}$  and 
	$$ \forall m\in [n,n+f(n)] \, \left(\langle p-u, p-z_{m+1}\rangle \leq \frac{1}{k+1}\right),$$
	where $\psi_{\chi}(k,f):=\xi_{\chi}(24N(\check{g}^{(R)}(0)+1)^2,f+1)$ with $R:= N^4(k+1)^2$, $\check{g}(m):=\max\{g(24N(m+1)^2), 24N(m+1)^2\}$ and $g(m):=f\left(\xi_{\chi}(m,f+1)\right)$.
\end{proposition}

\begin{proof}
	Let $k\in \N$ and a monotone function $f$ be given. By Proposition \ref{projection3} applied to $k$ and the function $g$, we get $n'\leq 24N\left(\check{g}^{(R)}(0)+1\right)^2$ and $p\in {\rm B}_{N}$  such that $\norm{\mathsf{J}(p)-p}\leq\frac{1}{g(n')+1}$ and
\begin{equation}\label{app_to g}
\forall q\in {\rm B}_{N} \, \left(\norm{\mathsf{J}(q)-q}\leq \frac{1}{n'+1}\to \langle p-u,p-q\rangle \leq \frac{1}{k+1}\right).
\end{equation}
	By \eqref{future_xi}, there is $n\leq \xi_{\chi}(n',f+1)\leq \psi_{\chi}(k,f)$ such that
	$$\forall m\in[n,n+f(n)+1]\, \left( \norm{\mathsf{J}(z_m)-z_m}\leq \frac{1}{n'+1}\right).$$
	Hence, $\forall m\in[n, n+f(n)] \, \left( \norm{\mathsf{J}(z_{m+1})-z_{m+1}}\leq \frac{1}{n'+1}\right)$. Since $(z_n)\subset {\rm B}_{N}$, by \eqref{app_to g} we conclude that
	$$\forall m\in [n,n+f(n)]\, \left(\langle p-u, p-z_{m+1}\rangle \leq \frac{1}{k+1}\right).$$
	Finally, by the monotonicity of the function $f$,
\begin{equation*}
\norm{\mathsf{J}(p)-p}\leq \frac{1}{g(n')+1}=\frac{1}{f(\xi_{\chi}(n',f+1))+1}\leq \frac{1}{f(n)+1}.\qedhere
\end{equation*}
\end{proof}

As mentioned in Subsection~\ref{Sectionoriginalproof}, the final step of the proof of Theorem~\ref{ThmYaoNoor} is an application of Lemma~\ref{LemmaXu}. There  $s_n=\norm{z_n -p}$, with $p$ the projection point of $u$ onto $S$. However, using approximations  to the projection point instead of the projection point itself, the inequality \eqref{xu_lem_main_ass} only holds with $s_n+v_n$ in place of $s_n$, for $(v_n)$ a certain sequence of errors.
The following result from \cite{PP(ta)} corresponds to a quantitative version of this statement.

\begin{lemma}[\cite{PP(ta)}]\label{lemmaqtXu1}
Let $(s_n)$ be a bounded sequence of non-negative real numbers and $D\in\N$ a positive upper bound on $(s_n)$. Consider sequences of real numbers $(\lambda_n)\subset (0,1)$, $(r_n)$, $(v_n)$ and $(\gamma_n)\subset [0, +\infty)$ and assume the existence of a monotone function $ L$ satisfying  $\sum_{i=1}^{L(k)}\lambda_i \geq k$, for all $k \in \N$. For natural numbers $k, n$ and $p$ assume
\begin{enumerate}[$(i)$]
\item $\forall m\in[n,p]\, \left(v_m\leq \frac{1}{4(k+1)(p+1)}\land r_m\leq \frac{1}{4(k+1)}\right)$.
\item $\forall m\in\N\, \left(\sum_{i=n}^{n+m}\gamma_i\leq \frac{1}{4(k+1)}\right)$.
\item $\forall m\in\N (s_{m+1}\leq (1-\lambda_m)(s_m+v_m)+\lambda_mr_m+\gamma_m).$
\end{enumerate}
Then
\[\forall m\in[\sigma(k,n),p]\, \left(s_m\leq \frac{1}{k+1}\right),\]
with $\sigma(k,n):=L\left(n+\lceil \ln(4D(k+1))\rceil\right)+1$.
\end{lemma}

A direct application of Lemma~\ref{lemmaqtXu1} gives the following result which is more suitable for our analysis.

\begin{lemma}\label{lemmaqtXu2}
Let $\Omega$ be a bounded subset of $H$. Let $(\lambda_n)\subset (0,1)$ be given and, for each $p\in\Omega$, consider the sequences of real numbers $(s_{n,p})$, $(v_{n,p})$, $(r_{n,p})$ and $(\gamma_{n,p})$ with $(s_{n,p})$, $(\gamma_{n,p}) \subset [0, +\infty)$ and such that, for all $p\in\Omega$,
\begin{equation*}
\forall m\in\N\, \left(s_{m+1,p}\leq (1-\lambda_m)(s_{m,p}+v_{m,p})+\lambda_mr_{m,p}+\gamma_{m,p}\right). 
\end{equation*}
For a natural number $D\in\N$ and monotone functions $ L$, $ G:\N\to\N$ and $\Psi:\N\times \N^{\N}\to\N$, assume that:
\begin{enumerate}[$(i)$]
\item $ L
$ is a rate of divergence for $\left(\sum\lambda_n\right)$, i.e.\  $\forall k \in \N \left(\sum_{i=1}^{L(k)}\lambda_i \geq k\right)$.
\item For all $p\in\Omega$, $D$ is a positive upper bound on $(s_{n,p})$.
\item For all $p\in\Omega$, $ G$ is a Cauchy rate for $\left(\sum\gamma_{n,p}\right)$, i.e.\ $\forall k,n \in \N \left(\sum_{i=G(k)+1}^{G(k)+n} \gamma_{i,p} \leq \frac{1}{k+1}\right) $.
\item $\forall k\in\N \,\mforall f:\N\to\N\,\exists p\in\Omega\,\exists n\leq \Psi(k,f)\,\forall m\in[n,f(n)]\, \left( v_{m,p}\leq \frac{1}{f(n)+1} \land r_{m,p}\leq \frac{1}{k+1}\right)$.
\end{enumerate}
Then, for any natural number $k$ and monotone function $f:\N\to\N$, there are $p\in\Omega$ and $n\leq \Theta(k,f)$ such that
\[\forall m\in [n,f(n)]\, \left( s_{m,p}\leq \frac{1}{k+1}\right),\]
where $\Theta(k,f)=\Theta(k,f,L,\Psi, G, D):= L\left(h(\Psi(4k+3,g))\right)+1$ with
$h(m)=\max\{m, G(4k+3)+1\}+\lceil \ln(4D(k+1))\rceil$ and $g(m):=4(k+1)\left(f(L(h(m))+1)+1\right)$.
\end{lemma}

\begin{proof}
Let $k\in\N$ and a monotone function $f$ be given. By condition $(iv)$, consider $p_0\in\Omega$ and $n_1\leq \Psi(4k+3,g)$ such that for $m\in [n_1,g(n_1)]$
\[v_{m,p_0}\leq \frac{1}{g(n_1)+1}\, \text{ and } \, r_{m,p_0}\leq \frac{1}{4(k+1)}.\]
Define $n_2:=\max\{n_1,  G(4k+3)+1\}$. By condition $(iii)$, for all $m\in\N$, $\sum_{i=n_2}^{n_2+m}\gamma_{m,p_0}\leq \frac{1}{4(k+1)}$.
We have $n_1\leq n_2$ and $$g(n_1)\geq f( L(h(n_1))+1)=f(\sigma(k,n_2)),$$ where $\sigma$ is as in Lemma~\ref{lemmaqtXu1}. Hence, for $m\in [n_2, f(\sigma(k,n_2))]$,
\[v_{m,p_0}\leq \frac{1}{4(k+1)(f(\sigma(k,n_2))+1)} \, \text{ and } \, r_{m,p_0}\leq \frac{1}{4(k+1)}.\]
We are in the conditions of Lemma~\ref{lemmaqtXu1} with $n=n_2$ and $p=f(\sigma(k,n_2))$, and so
\[\forall m\in[\sigma(k,n_2),f(\sigma(k,n_2))]\, \left(s_{m,p_0}\leq \frac{1}{k+1}\right).\]
Noticing that, by the monotonicity of $ L$, we have $\sigma(k,n_2)\leq \Theta(k,f)$, we conclude the proof.
\end{proof}

We recall that the last step in the proof by Yao and Noor is an application of Lemma~\ref{LemmaXu}. Similarly, in our quantitative analysis, the final step to prove metastability for \eqref{PPA} is an application of Lemma~\ref{lemmaqtXu2}. As such, we need to verify each of the conditions of that result (for a specific choice of parameters). Conditions $(i)$ and $(ii)$ are easy to check. The existence of a function $G$ as in condition $(iii)$ follows from a quantitative version of $(H_{\ref{H6}})$. The next result ensures that condition $(iv)$ holds.

\begin{lemma}\label{psi}
Assume the conditions of Proposition~\ref{removal_weakcompact}, and that there exist $c\in \N\setminus\{0\}$ satisfying $(Q_{\ref{ineqcn}})$ and $\mathcal{C}:\N \to \N$ a monotone function such that $\forall n \in \N (c_n \leq \mathcal{C}(n))$. For any $k\in\N$ and monotone function $f:\N\to\N$, there are $p\in {\rm B}_N$ and $n\leq \Psi_{\chi}(k,f)$ such that for all $m\in[n,f(n)]$,
$$v_{m,p}\leq \frac{1}{f(n)+1} \land r_{m,p}\leq \frac{1}{k+1},$$
where  $v_{m,p}=\norm{\mathsf{J}_{m}(p)-p}\left(\norm{\mathsf{J}_{m}(p)-p}+2\norm{z_m-p} \right)$, $r_{m,p}=2 \langle u-p,z_{m+1}-p\rangle$,  $\Psi_{\chi}(k,f):=\psi_{\chi}(2k+1, h)$, where $\psi_{\chi}$ is the function defined in Proposition~\ref{removal_weakcompact} and $h:\N\to\N$ is the monotone function defined by $$h(m):=\zeta\left( (1+4N)(f(m)+1)-1, f(m)\right),$$ with $\zeta$ as in Lemma~\ref{samefix-pt-set}.
\end{lemma}

\begin{proof}
Let $k\in\N$ and monotone $f$ be given. Applying Proposition~\ref{removal_weakcompact} to $2k+1$ and to the monotone function $h$ we obtain $p\in{\rm B}_N$ and $n\leq \psi_{\chi}(2k+1,h)$ such that
\begin{equation}\label{psi_a}
\norm{\mathsf{J}(p)-p}\leq \frac{1}{h(n)+1}
\end{equation}
and
\begin{equation}\label{psi_b}
\forall m\in [n, n+f(n)]\, \left( \langle u-p, z_{m+1}-p\rangle \leq \frac{1}{2(k+1)}\right).
\end{equation}
Clearly \eqref{psi_b} implies that for $m\in [n,f(n)]$ one has $r_{m,p}\leq \frac{1}{k+1}$.

Now, by \eqref{psi_a} $$\norm{\mathsf{J}(p)-p}\leq \frac{1}{\zeta\left((1+4N)(f(n)+1)-1,f(n)\right)+1}.$$
Hence, by \eqref{zeta_monotone}, for $m\leq f(n)$, $$\norm{\mathsf{J}_m(p)-p}\leq \frac{1}{(1+4N)(f(n)+1)}.$$
Also $\norm{\mathsf{J}_n(p)-p}\leq 1$ so, for $m\leq f(n)$,
$$v_{m,p}=\norm{\mathsf{J}_m(p)-p}\left(\norm{\mathsf{J}_m(p)-p}+2\norm{z_m-p}\right)\leq \frac{1+4N}{(1+4N)(f(n)+1)}=\frac{1}{f(n)+1},$$ which concludes the proof.
\end{proof}

\subsection{Rational approximation of the lim$\!$ sup}\label{SectionBypassACA}

In this section we show that the assumption of the existence of the $ \limsup $, as in Lemma~\ref{SuzukiLemma1}, can be replaced by a rational approximation. A detailed explanation on the origin of these lemmas is given in Section~\ref{Sectionlogic}.

The idea is that, by working with approximated notions, one can relax the properties of the $\limsup$ to something which is already satisfied by a suitable rational number. We start with the following easy result.

\begin{lemma}\label{lemmaratap}
Let $N \in \N$ and $(x_n)$ be a sequence of real numbers such that $\forall n \in \N (0 \leq x_n\leq N)$. Then 
\begin{equation}\label{lemmarationalaprox}
\begin{split}
\forall k,n \in \N \,\mforall f:\N \to \N  \, &\exists p<N(k+1)\\ &\left(\exists m \in [n, n+f(n)] \left(x_m \geq \frac{p}{k+1}\right) \wedge \forall m' \in [n,n+f(n)] \left(x_{m'} \leq \frac{p+1}{k+1}\right)\right).
\end{split}
\end{equation}
\end{lemma}

\begin{proof}
Suppose towards a contradiction that \eqref{lemmarationalaprox} does not hold. Then there exist $k,n \in \N$ and a monotone function $f$ such that for all $p<N(k+1)$ it holds that
\begin{equation}\label{hypreductio}
\forall m \in [n, n+f(n)] \left(x_m < \frac{p}{k+1}\right)\vee \exists m' \in [n,n+f(n)] \left(x_{m'} > \frac{p+1}{k+1}\right).
\end{equation}
This implies 
\begin{equation*}
\forall p<N(k+1)\left( A(p) \vee \neg A(p+1)\right),
\end{equation*}
where $A(p):\equiv \forall m \in [n,n+f(n)]\left(x_m<\frac{p}{k+1} \right)$. 
One easily shows by induction on $M \in \N$ that
\begin{equation*}
\forall M\left(\forall p \leq M \left(A(p) \vee \neg A(p+1)\right) \to \left(A(0) \vee \neg A(M+1)\right)\right).
\end{equation*} 
Hence, with $M=N(k+1)-1$ we conclude that
\begin{equation*}
\forall m \in [n,n+f(n)] \left(x_m< 0 \right) \vee \exists m \in [n,n+f(n)] \left(x_m \geq N \right).
\end{equation*}
Hence $\exists m \in [n,n+f(n)] \left(x_m \geq N\right)$. Now, by \eqref{hypreductio}, for $p=N(k+1)-1$ and the hypothesis, we have that for all $ m \in [n, n+f(n)]$ it holds that $x_m < \frac{N(k+1)-1}{k+1}=N-\frac{1}{k+1}$, which gives a contradiction. We conclude that \eqref{lemmarationalaprox} holds.
\end{proof}

The proof of Lemma~\ref{SuzukiLemma1} requires the following property of the $\limsup$
\begin{equation}\label{propertylimsup}
\forall k, M, t\in\N \,\exists m\geq M \,\forall n\geq m\, \left( x_{m+t} \geq \limsup x_n-\frac{1}{k+1}  \land x_n\leq \limsup x_n+\frac{1}{k+1}\right).
\end{equation}
We adapt Lemma~\ref{lemmaratap} to this property, obtaining a result that corresponds to a quantitative version of \eqref{propertylimsup}.
 
\begin{lemma}\label{lemmarationalapprox2}
Let $N \in \N$ and $(x_n)$ be a sequence of real numbers such that $\forall n \in \N (0 \leq x_n\leq N)$. Let $k,M,t \in \N$ and $f:\N \to \N$ be monotone, let $P:=N(k+1)$. For $i\in\{0, \dots,P\}$ define $n_i=M+it$ and $r_i:=\begin{cases}0, & i=P\\ t+r_{i+1}+f(n_{i+1}+r_{i+1}), &   i<P.\end{cases}$ Then 
\begin{equation}
\exists p<P\, \exists m \in [M,\theta]\left( x_{m+t} \geq \frac{p}{k+1} \wedge \forall n\in [m,m+f(m)]\left( x_n \leq \frac{p+1}{k+1}\right)\right),
\end{equation}
where $\theta=\theta(k,M,t,N,f):=M+(P-1)t+r_0$.
\end{lemma}

\begin{proof}
Let $k,M,t \in \N$ and $f$ be a given monotone function. We define, for each $i\leq P$, the monotone functions $g_i:=\lambda m.r_i$.
We apply \eqref{lemmarationalaprox} with $k=k$, $f=g_i$ and $n=n_i$, for $i \leq P$. Then, we find, for each $i \leq P$, $m_i \in [n_i, n_i+r_i]$ and $p_i <P$ such that $x_{m_i}\geq \frac{p_i}{k+1}$ and $\forall n \in [n_i,n_i+r_i] \left(x_{n} \leq \frac{p_i+1}{k+1} \right)$.
Now, there exists $i_0< P$ such that $p_{i_0}\leq p_{i_0+1}$, otherwise there would be a sequence of length $P+1$ of natural numbers  such that $p_{P}<p_{P-1}<\dots <p_1<p_0< P$, which is absurd.  
Define the natural numbers $m:=m_{i_0+1}-t$ and $p:=p_{i_0+1}$. Clearly $m \in [M,\theta]$ and $p<P$. We have that $x_{m+t}\geq \frac{p}{k+1}$. To conclude the result it is enough to show that $[m,m+f(m)]\subseteq[n_{i_0},n_{i_0}+r_{i_0}]$. Indeed, we would get, for $n \in [m,m+f(m)]$ that $x_n \leq \frac{p_{i_0}+1}{k+1}\leq \frac{p_{i_0+1}+1}{k+1}=\frac{p+1}{k+1}$.
We have that $m=m_{i_0+1}-t\geq n_{i_0+1}-t=n_{i_0}$, and since $f$ is monotone, $m+f(m)\leq m_{i_0+1}+f(m_{i_0+1})\leq n_{i_0+1}+r_{i_0+1}+f(n_{i_0+1}+r_{i_0+1})=n_{i_0}+t+r_{i_0+1}+ f(n_{i_0+1}+r_{i_0+1})=n_{i_0}+r_{i_0}$. Hence $[m,m+f(m)]\subseteq[n_{i_0},n_{i_0}+r_{i_0}]$.
\end{proof}

\section{Quantitative analysis}\label{SectionQuantitative}

In this section we carry out the quantitative analysis of Theorem~\ref{ThmYaoNoor}. In Subsection~\ref{SectionMetastability} we obtain intermediate results regarding asymptotic regularity and metastability depending on an additional condition. This additional condition is studied in Subsection~\ref{SectionMainLemmas} through the analysis of Suzuki's lemmas (Lemmas~\ref{SuzukiLemma1} and \ref{SuzukiLemma2}). In Subsection~\ref{Sectionrevisited} we prove our main result establishing the metastability for \eqref{PPA}.

We start our quantitative analysis of Theorem~\ref{ThmYaoNoor} by giving quantitative versions of the hypothesis of the theorem. We assume that there exist $a,c \in \N\setminus \{0\}$ and monotone functions $\ell,L, \Gamma,E: \N \to \N$ such that 
\begin{enumerate}[($Q_1$)]
\item\label{hyp1} $\forall k \in \N \,\forall n \geq \ell(k) \left(\lambda_n \leq \frac{1}{k+1}\right)$.
\item\label{ratediv} $\forall k \in \N \left(\sum_{i=1}^{L(k)}\lambda_i \geq k\right)$.
\item\label{ineqgamma} $\forall m\in \N\left(\frac{1}{a}\leq \gamma_m \leq1-\frac{1}{a} \right)$. 
\item\label{ineqcn} $\forall n \in \N\left(c_n \geq \frac{1}{c} \right)$.
\item\label{Q5} $\forall k\in \N \, \forall n\geq \Gamma(k)\left(\left \vert c_{n+1}-c_n \right\vert\leq \frac{1}{k+1}\right)$.
\item\label{ineqerror1} $\forall k \in \N \,\forall n \in \N \left(\sum_{i=E(k)+1}^{E(k)+n}\norm{e_i}\leq \frac{1}{k+1} \right)$. 
\end{enumerate}

The conditions $(Q_{\ref{hyp1}})-(Q_{\ref{ineqerror1}})$ are quantitative versions of, respectively, the hypothesis ($H_{\ref{H1}}$)$-$($H_{\ref{H6}}$). Indeed, condition ($Q_{\ref{hyp1}}$) states that $\ell$ is a rate of convergence for the sequence $(\lambda_n)$; condition ($Q_{\ref{ratediv}}$) postulates that $L$ is a rate of divergence for $\left(\sum \lambda_n\right)$; condition ($Q_{\ref{ineqgamma}}$) is the quantitative version of ($H_{\ref{H3}}$) together with the fact that $(\gamma_n) \subset (0,1)$; condition ($Q_{\ref{ineqcn}}$) expresses the fact that the terms of the sequence $(c_n)$ are above some positive quantity; condition ($Q_{\ref{Q5}}$) states that $\Gamma$ is a rate of convergence for the difference of terms of the sequence $(c_n)$ and condition $(Q_{\ref{ineqerror1}})$ expresses quantitatively that the sequence of the partial sums of the errors $e_n$ is a Cauchy sequence with Cauchy rate $E$. 

In our main result (Theorem~\ref{theoremyaonoor2})  we compute an explicit bound on the metastability of the iteration \eqref{PPA} under the assumptions $(Q_{\ref{hyp1}})-(Q_{\ref{ineqerror1}})$. 

\subsection{Metastability of the \texorpdfstring{$\mPPA$}{mPPA}}\label{SectionMetastability}
We show an intermediate metastability result depending on an additional condition \eqref{hypchi} in Theorem~\ref{theoremyaonoor}. Moreover, Proposition~\ref{lemmaintermediate} and Corollary~\ref{Corasympreg} give quantitative information on the asymptotic regularity of the iteration.
We start by showing that the sequence $(z_n)$ generated by \eqref{PPA} is bounded and give in \eqref{limwm-zm} the computational information corresponding to $\limsup\left(\norm{w_{n+1}-w_n}-\norm{z_{n+1}-z_n}\right)\leq 0$.
\begin{lemma}\label{lemmamain1}
Let $(z_n)$ be generated by \eqref{PPA}. Assume that there exist $a,c \in \N\setminus\{0\}$ and monotone functions $\ell, \Gamma,E$ such that $(Q_{\ref{hyp1}}), (Q_{\ref{ineqgamma}})-(Q_{\ref{ineqerror1}})$ hold.  Let $N_1, N_2, N_3 \in \N$ be such that $N_1 \geq \norm{u}$, $N_2 \geq \sum_{i=0}^{E(0)}\norm{e_i}+1$, and for some $s \in S$ one has $N_3 \geq \max\{\norm{u-s},\norm{z_0-s}\}$. Then  $\norm{z_n-s}\leq N_0$, where $N_0:= N_2+ N_3$. Moreover, with $z_{n+1}=\gamma_nz_n+(1-\gamma_n)w_n$, we have $\norm{w_n-s}\leq 2aN_0$ and
\begin{equation}\label{limwm-zm}
\forall k \in \N \, \forall n \geq \nu(k)\,\left(\norm{w_{n+1}-w_n}-\norm{z_{n+1}-z_n}\leq \frac{1}{k+1} \right),
\end{equation}
where $\nu(k):=\max\{\Gamma(10acN_0(k+1)),\ell(10a(N_0+N_1+N_3)(k+1)),E(5a(k+1))+1\}$.
\end{lemma}
\begin{proof}
Observe that $\sum_{i=0}^{n}\norm{e_i}\leq N_2$ for all $n \in \N$.
By the fact that $\lambda_n+\gamma_n+\delta_n=1$, for all $n \geq 0$ and observing that each resolvent $\mathsf{J}_c$ is nonexpansive, we have
\begin{equation*}
\begin{split}
\norm{z_{n+1}-s}&=\norm{\lambda_n(u-s)+\gamma_n(z_n-s)+\delta_n(\mathsf{J}_{n}(z_n)-s)+e_n}\\
&\leq \lambda_n\norm{u-s}+\gamma_n\norm{z_n-s}+\delta_n\norm{z_n-s}+\norm{e_n}\\
&=\lambda_n\norm{u-s}+(1-\lambda_n)\norm{z_n-s}+\norm{e_n}. 
\end{split}
\end{equation*}
One easily shows by induction on $n\in \N$ that $\norm{z_n-s}\leq \max\{\norm{u-s},\norm{z_0-s}\}+\sum_{i=0}^{n-1}\norm{e_i}\leq N_0$, from which we deduce that  $(z_n)$ is bounded.
We have that $\norm{w_n-s}\leq 2aN_0$. Indeed,  by ($Q_{\ref{ineqgamma}}$) we have 
\begin{equation*}
\norm{w_n -s}= \frac{\norm{z_{n+1}-s-\gamma_n(z_n-s)}}{1-\gamma_n}\leq \frac{2N_0}{1-\gamma_n}\leq 2aN_0.
\end{equation*}
We have 
\begin{equation*}
\begin{split}
w_{m+1}-w_m&=\frac{\lambda_{m+1}u+\delta_{m+1}\mathsf{J}_{m+1}(z_{m+1})+e_{m+1}}{1-\gamma_{m+1}}-\frac{\lambda_{m}u+\delta_{m}\mathsf{J}_{m}(z_{m})+e_m}{1-\gamma_{m}}\\
&=\left(\frac{\lambda_{m+1}}{1-\gamma_{m+1}}-\frac{\lambda_{m}}{1-\gamma_{m}} \right)u+\frac{\delta_{m+1}}{1-\gamma_{m+1}}\left(\mathsf{J}_{m+1}(z_{m+1})-\mathsf{J}_{m}(z_m)\right)\\
&+ \left(\frac{\delta_{m+1}}{1-\gamma_{m+1}} -\frac{\delta_{m}}{1-\gamma_{m}}\right)\mathsf{J}_{m}(z_m)+ \frac{e_{m+1}}{1-\gamma_{m+1}}-\frac{e_{m}}{1-\gamma_{m}}.
\end{split}
\end{equation*}
We claim that 

\begin{equation}\label{ineqJc}
\norm{\mathsf{J}_{m+1}(z_{m+1})-\mathsf{J}_{m}(z_m)}\leq \norm{z_{m+1}-z_m}+ 2cN_0|c_{m+1}-c_m|.
\end{equation}

To prove the claim observe that for every $n,m \in \N$ it holds that $ \norm{\mathsf{J}_{m}(z_n)-s}\leq \norm{z_n-s}\leq N_0$.
If $c_m\leq c_{m+1}$, by the resolvent identity we have
\begin{equation*}
\begin{split}
\norm{\mathsf{J}_{m+1}(z_{m+1})-\mathsf{J}_{m}(z_m)}&=\norm{\mathsf{J}_{m}\left(\frac{c_{m}}{c_{m+1}}z_{m+1}+\left(1-\frac{c_{m}}{c_{m+1}}\right)\mathsf{J}_{{m+1}}(z_{m+1})\right)-\mathsf{J}_{{m}}(z_{m})}\\
&\leq \frac{c_{m}}{c_{m+1}}\norm{z_{m+1}-z_m}+\left(1-  \frac{c_{m}}{c_{m+1}} \right)\norm{\mathsf{J}_{{m+1}}(z_{m+1})-z_m}\\
&\leq \norm{z_{m+1}-z_m}+c|c_{m+1}-c_m|\norm{\mathsf{J}_{{m+1}}(z_{m+1})-z_m}\\
&\leq  \norm{z_{m+1}-z_m}+2cN_0|c_{m+1}-c_m|.
\end{split}
\end{equation*}
If $c_{m+1}<c_m$, again by the resolvent identity we have
\begin{equation*}
\begin{split}
\norm{\mathsf{J}_{{m}}(z_{m})-\mathsf{J}_{{m+1}}(z_{m+1})}&=\norm{\mathsf{J}_{{m+1}}\left(\frac{c_{m+1}}{c_{m}}z_m+\left(1-\frac{c_{m+1}}{c_{m}}\right)\mathsf{J}_{{m}}(z_{m})\right)-\mathsf{J}_{{m+1}}(z_{m+1})}\\
&\leq \frac{c_{m+1}}{c_{m}}\norm{z_{m+1}-z_m}+\left(1-  \frac{c_{m+1}}{c_{m}} \right)\norm{\mathsf{J}_{{m}}(z_{m})-z_{m+1}}\\
&\leq \norm{z_{m+1}-z_m}+c|c_{m+1}-c_m|\norm{\mathsf{J}_{{m}}(z_{m})-z_{m+1}}\\
&\leq  \norm{z_{m+1}-z_m}+2cN_0|c_{m+1}-c_m|.
\end{split}
\end{equation*}
Hence \eqref{ineqJc} holds. Then 
\begin{equation}\label{ineqlim0}
\begin{split}
\norm{w_{n+1}-w_n}&- \norm{z_{m+1}-z_m}\leq \left\lvert\frac{\lambda_{m+1}}{1-\gamma_{m+1}}-\frac{\lambda_{m}}{1-\gamma_{m}} \right\rvert {\norm{u}}+\left \vert\frac{\delta_{m+1}}{1-\gamma_{m+1}}-1 \right \vert\norm{z_{m+1}-z_m}\\
&+\left\vert \frac{\delta_{m+1}}{1-\gamma_{m+1}}\right \vert2cN_0|c_{m+1}-c_m|+ \left\vert\frac{\delta_{m+1}}{1-\gamma_{m+1}}-\frac{\delta_{m}}{1-\gamma_{m}} \right \vert\norm{\mathsf{J}_{m}(z_m)}+\frac{\norm{e_{m+1}}}{1-\gamma_{m+1}}+\frac{\norm{e_{m}}}{1-\gamma_{m}}.
\end{split}
\end{equation}

Let $k \in \N$ and $m \geq \nu(k)$. We will see that each of the terms in \eqref{ineqlim0} is less than or equal to $\frac{1}{5(k+1)}$.

 Since $\ell$ satisfies $(Q_{\ref{hyp1}})$, we have that, for $m \geq \ell(10a(N_0+N_1+N_3)(k+1))$ 
\begin{equation}\label{ineqlim1}
\begin{split}
\left\vert\frac{\lambda_{m+1}}{1-\gamma_{m+1}}-\frac{\lambda_{m}}{1-\gamma_{m}} \right\vert{\norm{u}}&\leq \left(\frac{\lambda_{m+1}}{1-\gamma_{m+1}}+\frac{\lambda_{m}}{1-\gamma_{m}} \right) N_1\\
&\leq\left(\lambda_{m+1}+\lambda_{m} \right)aN_1  \\
&\leq \frac{2aN_1}{10a(N_0+N_1+N_3)(k+1)}\\
&\leq\frac{1}{5(k+1)}
\end{split}
\end{equation}
and
\begin{equation}\label{ineqlim2}
\begin{split}
\left \vert 1-\frac{\delta_{m+1}}{1-\gamma_{m+1}} \right\vert\norm{z_{m+1}-z_m}&\leq \left \vert \frac{1-\gamma_{m+1}-\delta_{m+1}}{1-\gamma_{m+1}}\right \vert2N_0\\
&\leq \lambda_{m+1}2aN_0\\
&\leq \frac{2aN_0}{10a(N_0+N_1+N_3)(k+1)}\\
&\leq\frac{1}{5(k+1)}.
\end{split}
\end{equation}

Observe that $\norm{\mathsf{J}_m(z_m)}\leq \norm{\mathsf{J}_m(z_m)-s}+\norm{u-s}+\norm{u}\leq N_0+N_1+N_3$. We then have
\begin{equation}\label{ineqlim4}
\begin{split}
\left\vert\frac{\delta_{m+1}}{1-\gamma_{m+1}}-\frac{\delta_{m}}{1-\gamma_{m}} \right \vert\norm{\mathsf{J}_{m}(z_m)}&\leq \left(\frac{\lambda_{m+1}}{1-\gamma_{m+1}}+\frac{\lambda_{m}}{1-\gamma_{m}} \right)(N_0+N_1+N_3)\\
&\leq \frac{2a(N_0+N_1+N_3)}{10a(N_0+N_1+N_3)(k+1)}=\frac{1}{5(k+1)}.
\end{split}
\end{equation}

Since $\Gamma$ satisfies $(Q_{\ref{Q5}})$, for $m \geq \Gamma(10acN_0(k+1))$ it holds that
\begin{equation}\label{ineqlim3}
\left\vert \frac{\delta_{m+1}}{1-\gamma_{m+1}}\right \vert2cN_0|c_{m+1}-c_m|\leq \frac{2acN_0}{10acN_0(k+1)}=\frac{1}{5(k+1)}.
\end{equation}

Since $E$ satisfies ($Q_{\ref{ineqerror1}}$), for  $m \geq E(5a(k+1))+1$ we have
\begin{equation}\label{ineqlim5}
\frac{\norm{e_{m+1}}}{1-\gamma_{m+1}}+\frac{\norm{e_{m}}}{1-\gamma_{m}}\leq a \left(\norm{e_m}+\norm{e_{m+1}} \right)\leq a\left(\sum_{i=E(5a(k+1))+1}^{m+1}\norm{e_i}\right)\leq \frac{a}{5a(k+1)+1}\leq \frac{1}{5(k+1)}.
\end{equation}
Combining \eqref{ineqlim0}-\eqref{ineqlim5} we conclude that \eqref{limwm-zm} holds.
\end{proof}

\begin{definition}
Let $(z_n),(w_n)$ be sequences in $H$. We say that a monotone function $\chi: \N \times \N^{\N} \to \N$ satisfies \eqref{hypchi} if
\begin{equation}\label{hypchi}\tag{$Q_S$}
\forall k \in \N \,\mforall f:\N \to \N \,\exists n \leq \chi(k,f)\,\forall m \in [n,n+f(n)]\left(\norm{w_{m}-z_m}\leq \frac{1}{k+1} \right).
\end{equation}
\end{definition}

\begin{remark}
The hypothesis \eqref{hypchi} corresponds to the quantitative information from Suzuki's lemma (Lemma~\ref{SuzukiLemma2}). With this assumption we will compute in Theorem~\ref{theoremyaonoor} metastability for \eqref{PPA} as well as some intermediate results regarding asymptotic regularity (Proposition~\ref{lemmaintermediate} and Corollary~\ref{Corasympreg}). An explicit function satisfying \eqref{hypchi} is computed in Remark~\ref{witQS}, using Lemma~\ref{LemmaSuzuki2}. 
\end{remark}

The next two results give quantitative information on asymptotic regularity for the sequence $(z_n)$.
\begin{proposition}\label{lemmaintermediate}
Let $(z_n)$ be generated by \eqref{PPA}. Assume that there exist $a,c \in \N \setminus\{0\}$ and monotone functions $\ell,E$ such that $(Q_{\ref{hyp1}}), (Q_{\ref{ineqgamma}}), (Q_{\ref{ineqcn}})$ and $(Q_{\ref{ineqerror1}})$ hold. Let $N_1, N_2, N_3 \in \N$ be such that $N_1 \geq \norm{u}$, $N_2 \geq \sum_{i=0}^{E(0)}\norm{e_i}+1$, and for some $s \in S$ one has $N_3 \geq \max\{\norm{u-s},\norm{z_0-s}\}$. Define $N_0:= N_2+ N_3$. Let $(w_n)$ be such that $z_{n+1}=\gamma_nz_n+(1-\gamma_n)w_n$ and assume that there is a monotone function $\chi$ satisfying \eqref{hypchi}.
 Then
\begin{enumerate}[$(i)$]
\item\label{limzn+1-zn}
$\forall k \in \N \,\mforall f:\N \to \N \,\exists n \leq \chi(k,f)\,\forall m \in [n,n+f(n)]\left(\norm{z_{m+1}-z_m}\leq \frac{1}{k+1} \right)$;
\item\label{limzn-J}
$\forall k \in \N \,\mforall f:\N \to \N \,\exists n \leq \max\left\{\mu(k),\chi\left(2a(k+1),\tilde{f}_k\right)\right\}\,\forall m \in [n,n+f(n)]\left(\norm{\mathsf{J}_{m}(z_m)-z_m}\leq \frac{1}{k+1} \right)$;
\item \label{limzn-Jc}$\forall k \in \N \,\mforall f:\N \to \N \,\exists n \leq \xi_{\chi}(k,f)\,\forall m \in [n,n+f(n)]\left(\norm{\mathsf{J}(z_m)-z_m}\leq \frac{1}{k+1} \right)$;
\end{enumerate}
where $\mu(k):=\max\{\ell(4a(k+1)(N_0 +N_3 )),E(4a(k+1)+1\}$ and $\xi_{\chi}(k,f):=\max\{\mu(2k+1),\chi(4a(k+1),\tilde{f}_{2k+1})\}$, with $\tilde{f}_k:\N \to \N$ the monotone function defined by $\tilde{f}_k(m)=\mu(k)+f(\max\{\mu(k),m \})$.
\end{proposition}
\begin{proof}
\eqref{limzn+1-zn}.
The result is a consequence of the fact that $$\norm{z_m-z_{m+1}}=\norm{z_m-\gamma_mz_m-(1-\gamma_m)w_m}=(1-\gamma_m)\norm{z_m-w_m}\leq \norm{z_m-w_m}.$$ 

\eqref{limzn-J}.
Observe that 
\begin{equation*}
\begin{split}
\norm{\mathsf{J}_{m}(z_m)-z_m}&\leq \norm{\mathsf{J}_{m}(z_m)-z_{m+1}}+\norm{z_{m+1}-z_m}\\
&\leq \norm{z_{m+1}-z_m}+\lambda_m\norm{\mathsf{J}_{m}(z_m)-u}+\gamma_m\norm{\mathsf{J}_{m}(z_m)-z_m}+\norm{e_m}.
\end{split}
\end{equation*}
Then
\begin{equation}
\norm{\mathsf{J}_{m}(z_m)-z_m}\leq \frac{\norm{z_{m+1}-z_m}}{1-\gamma_m}+\frac{\lambda_m\norm{\mathsf{J}_{m}(z_m)-u}}{1-\gamma_m}+\frac{\norm{e_m}}{1-\gamma_m}.
\end{equation}
We have that 
\begin{equation}\label{rate1}
\forall k \in \N \,\forall m \geq \mu(k) \left( \frac{\lambda_m\norm{\mathsf{J}_{m}(z_m)-u}}{1-\gamma_m}+\frac{\norm{e_m}}{1-\gamma_m}\leq\frac{1}{2(k+1)}\right).
\end{equation}
Indeed, for $m \geq \mu(k)$ we have that
\begin{equation*}
\begin{split}
\frac{\lambda_m\norm{\mathsf{J}_{m}(z_m)-u}}{1-\gamma_m}+\frac{\norm{e_m}}{1-\gamma_m}&\leq \frac{a\left(N_0+N_3 \right)}{4a(k+1)\left(N_0+N_3 \right)}+a\left( \sum_{i=E(4a(k+1))+1}^{m}\norm{e_i}\right)\\
&\leq \frac{1}{4(k+1)}+\frac{1}{4(k+1)}=\frac{1}{2(k+1)}.
\end{split}
\end{equation*}
Applying Part~\eqref{limzn+1-zn} to $2a(k+1)$ and $\tilde{f}_k$ we find $n'\leq \chi(2a(k+1),\tilde{f}_k)$ such that
\begin{equation}\label{meta1}
\forall m \in [n',n'+\tilde{f}_k(n')]\left(\frac{\norm{z_{m+1}-z_m}}{1-\gamma_m}\leq \frac{a}{(2a(k+1))+1}\leq \frac{1}{2(k+1)}\right).
\end{equation}
Put $n=\max\{\mu(k),n' \}$. Now $[n,n+f(n)]\subseteq [n',n'+\tilde{f}_k(n')]$ because clearly $n' \leq n$ and $n+f(n)\leq n'+\mu(k)+f(\max\{\mu(k),n' \})=n'+\tilde{f}_k(n')$. Then, from \eqref{rate1} and \eqref{meta1} we conclude that Part~\eqref{limzn-J} holds.

\eqref{limzn-Jc}.
By Lemma~\ref{lemmaresolvineq} and ($Q_{\ref{ineqcn}}$)  we have $\norm{\mathsf{J}(z_m)-z_m}\leq 2 \norm{\mathsf{J}_{m}(z_m)-z_m}$. Hence,  Part~\eqref{limzn-Jc} follows from Part~\eqref{limzn-J}. 
\end{proof}

If \eqref{hypchi} is satisfied with a rate of convergence, i.e.\ when $\chi$ does not depend on $f$, then the properties~\eqref{limzn+1-zn}--\eqref{limzn-Jc} in Proposition~\ref{lemmaintermediate} also hold with rates of convergence.
\begin{corollary}\label{Corasympreg}
Assume the conditions of Proposition~\ref{lemmaintermediate}. Assume also that $\chi(k,f)=\chi(k)$, for all $f$, i.e.
\begin{equation}
\forall k \in \N \,\forall m \geq \chi(k)\left(\norm{w_{m}-z_m}\leq \frac{1}{k+1} \right).
\end{equation}
Then
\begin{enumerate}[$(i)$]
\item
$\forall k \in \N \,\forall m \geq \chi(k)\left(\norm{z_{m+1}-z_m}\leq \frac{1}{k+1} \right)$;
\item
$\forall k \in \N \,\forall m \geq \max\{\mu(k),\chi(2a(k+1)\}\left(\norm{\mathsf{J}_{m}(z_m)-z_m}\leq \frac{1}{k+1} \right)$;
\item $\forall k \in \N\,\forall m \geq \xi_{\chi}(k)\left(\norm{\mathsf{J}(z_m)-z_m}\leq \frac{1}{k+1} \right)$;
\end{enumerate}
where $\xi_{\chi}(k):=\max\{\mu(2k+1),\chi(4a(k+1))\}$.
\end{corollary}

Under the additional condition \eqref{hypchi} we have the following result concerning the metastability of \eqref{PPA}.

\begin{theorem}\label{theoremyaonoor}
Let $(z_n)$ be generated by \eqref{PPA}. Assume that there exist $a,c \in \N \setminus\{0\}$ and monotone functions $\ell,L,E$ such that $(Q_{\ref{hyp1}})-(Q_{\ref{ineqcn}})$ and $(Q_{\ref{ineqerror1}})$ hold. Let $\mathcal{C}:\N\to\N$ be a monotone function such that $c_n \leq\mathcal{C}(n)$, for all $n \in \N$.  Let $ N_2, N_3 \in \N$ be such that $N_2 \geq \sum_{i=0}^{E(0)}\norm{e_i}+1$, and for some $s \in S$ one has $N_3 \geq \max\{\norm{u-s},\norm{z_0-s}\}$. Define $N:=\max\{2N_3,N_2+N_3\}$.
Let $(w_n)$ be such that $z_{n+1}=\gamma_nz_n+(1-\gamma_n)w_n$ and assume that there is a monotone function $\chi$ satisfying \eqref{hypchi}.
Then
\begin{equation*}
\forall k \in \N \,\mforall f : \N \to \N \,\exists n \leq \phi_{\chi}(k,f) \forall i,j \in [n,n+f(n)] \left(\norm{z_i-z_j}\leq \frac{1}{k+1} \right),
\end{equation*}
where $\phi_{\chi}(k,f):=\Theta(4(k+1)^2-1,\lambda m\ldotp(m+f(m)),L,\Psi_{\chi}, G,4N^2)$, $G(k)=E(M_2(k+1))$, $\Theta$ is as in Lemma~\ref{lemmaqtXu2}, $\Psi_{\chi}$ as in Lemma~\ref{psi}, $M_1:=3 N_2+4N$ and $M_2=M_1+2(N_3+N)$.
\end{theorem}

\begin{remark}
Note that $\Psi_{\chi}=\Psi_{\chi}[N_2,N_3,a,c ,\ell,\mathcal{C}]$, i.e.\ the functional $\Psi_{\chi}$ depends on the value of $N_2$, $N_3$, $a$ and $c$, as well as on the functions $\ell$ and $\mathcal{C}$ -- and obviously on the functional $\chi$.
\end{remark}

\begin{proof}[Proof of Theorem~\ref{theoremyaonoor}]
Let $p \in {\rm B}_N$.
Then
\begin{equation*}
\begin{split}
\norm{z_{m+1}-p}^{2}&\leq \left(\norm{z_{m+1}-p-e_m}+\norm{e_m} \right)^2\\
& = \norm{z_{m+1}-p-e_m}^2+\norm{e_m}\left(\norm{e_m}+2\norm{z_{m+1}-p-e_m} \right)\\
& \leq \norm{z_{m+1}-p-e_m}^2+M_1\norm{e_m}\\
& = \norm{z_{m+1}-p-e_m-\lambda_m(u-p)+\lambda_m{u-p}}^2+M_1\norm{e_m}\\
& \leq  \norm{z_{m+1}-p-e_m-\lambda_m(u-p)}^2+2 \lambda_m \langle u-p,z_{m+1}-p-e_m \rangle+M_1\norm{e_m}\\
& \leq \norm{\gamma_{m}(z_{m}-p)+\delta_m\left(\mathsf{J}_{m}(z_m)-p \right)}^2+2 \lambda_m \langle u-p,z_{m+1}-p \rangle+ \norm{e_m}\left(M_1+2\lambda_m\norm{u-p} \right)\\
& \leq \left(\gamma_{m}\norm{(z_{m}-p)}+\delta_m\norm{\mathsf{J}_{m}(z_m)-p}\right)^2+2 \lambda_m \langle u-p,z_{m+1}-p \rangle \\& \,\,\,\,\, + \norm{e_m}\left(M_1+2\lambda_m\norm{u-p} \right)\\
& \leq  \left(\gamma_{m}\norm{(z_{m}-p)}+\delta_m\norm{\mathsf{J}_{m}(z_m)-\mathsf{J}_{m}(p)}+\delta_m\norm{\mathsf{J}_{m}(p)-p}\right)^2+2 \lambda_m \langle u-p,z_{m+1}-p \rangle \\& \,\,\,\,\, + \norm{e_m}\left(M_1+2\lambda_m\norm{u-p} \right)\\
& \leq  \left((1-\lambda_{m})\norm{(z_{m}-p)}+(1-\lambda_{m})\norm{\mathsf{J}_{m}(p)-p}\right)^2+2 \lambda_m \langle u-p,z_{m+1}-p \rangle \\& \,\,\,\,\, + \norm{e_m}\left(M_1+2\lambda_m\norm{u-p} \right)\\
& \leq  (1-\lambda_{m})\norm{z_{m}-p}^2+(1-\lambda_{m})\norm{\mathsf{J}_{m}(p)-p}\left(\norm{\mathsf{J}_{m}(p)-p}+2\norm{z_m-p} \right)\\& \,\,\,\,\,+2 \lambda_m \langle u-p,z_{m+1}-p\rangle + \norm{e_m}\left(M_1+2\lambda_m\norm{u-p} \right).
\end{split}
\end{equation*}

Then, for all $m \in \N$
\begin{equation*}
s_{m+1,p} \leq  (1-\lambda_{m})s_{m,p}+(1-\lambda_{m})v_{m,p}+ \lambda_m r_{m,p} +\gamma_{m,p},
\end{equation*}

where $s_{m,p}=\norm{z_m-p}^2$, $v_{m,p}=\norm{\mathsf{J}_{m}(p)-p}\left(\norm{\mathsf{J}_{m}(p)-p}+2\norm{z_m-p} \right)$, $r_{m,p}=2 \langle u-p,z_{m+1}-p\rangle$ and $\gamma_{m,p}=\norm{e_m}\left(M_1+2\lambda_m\norm{u-p} \right)$.

We verify that the conditions of Lemma~\ref{lemmaqtXu2} are satisfied with $\Omega={\rm B}_N$, $D=4N^2$, $\Psi_{\chi}$ as in Lemma~\ref{psi} and 
$G(k)=E(M_2(k+1))$.

 The first condition holds by hypothesis. Since $\norm{z_n-p}\leq 2N$, the second condition is true with $D=(2N)^2$.
For the third condition, using ($Q_{\ref{ineqerror1}}$), and the fact that $M_2 \geq M_1+2\lambda_m\norm{u-p}$ we have 
\begin{equation*}
\begin{split}
\sum_{i=G(k)+1}^{G(k)+n}\gamma_{i,p}&=\sum\limits_{i=E(M_2(k+1))+1}^{E(M_2(k+1))+n}\norm{e_m}\left(M_1+2\lambda_m\norm{u-p} \right)\\
& \leq \sum_{i=E(M_2(k+1))+1}^{E(M_2(k+1))+n}M_2\norm{e_m}\\
&\leq \frac{M_2}{M_2(k+1) +1}\leq \frac{1}{k+1}.
\end{split}
\end{equation*}
Finally, by Lemma~\ref{psi} the fourth condition of Lemma~\ref{lemmaqtXu2} is also verified.

By Lemma~\ref{lemmaqtXu2} we conclude that
\begin{equation}\label{finaleq}
\forall k\in\N \,\mforall f:\N\to\N \,\exists p\in {\rm B}_N\,\exists n\leq \Theta(k,f,L,\Psi_{\chi}, G,4N^2)\,\forall m\in[n,f(n)]\, \left(\|z_m-p\|^2\leq \frac{1}{k+1}\right).
\end{equation}
Let $k\in\N$ and a monotone function $f$ be given. By \eqref{finaleq} applied to $4(k+1)^2-1$ and to the function $\lambda m \ldotp (m+f(m))$, we find $p\in {\rm B}_N$ and $n\leq \phi_{\chi}(k,f)$ such that for $m\in[n,n+f(n)]$,
$$\|z_m-p\|^2\leq \frac{1}{4(k+1)^2}.$$
Hence, for $m\in[n,n+f(n)]$, $\|z_m-p\|\leq \frac{1}{2(k+1)}$ and with $i,j\in[n,n+f(n)]$, we have
$$\|z_i-z_j\|\leq \|z_i-p\|+\|z_j-p\|\leq \frac{1}{k+1},$$
which concludes the proof.
\end{proof}


\subsection{A quantitative version of Suzuki's lemmas}\label{SectionMainLemmas}

We now turn to the two lemmas by Suzuki required in the original proof. The next result is a partial quantitative version of Lemma~\ref{SuzukiLemma1}, which is enough for the quantitative version of Lemma~\ref{SuzukiLemma2} given in Lemma~\ref{LemmaSuzuki2}. As discussed in Remark~\ref{witQS} below, the latter allows us to obtain a concrete functional satisfying the condition \eqref{hypchi}.

\begin{lemma}\label{lemmaSuzuki1}
Let $(z_n),(w_n)$ be sequences in a normed space $X$. Let $(\alpha_n) \subset [0,1]$ be a sequence of real numbers and $a \in \N \setminus \{0\}$ be such that
\begin{equation}\label{ineq1}
\forall n \geq a \left(\alpha_n \leq 1 - \frac{1}{a}\right).
\end{equation}
Suppose that $z_{n+1}=\alpha_nw_n+(1-\alpha_n)z_n$, for all $n \in \N$ and that there exists a monotone function $\nu: \N \to \N$ satisfying
\begin{equation}\label{ineq2}
\forall k \in \N \, \forall n \geq \nu(k) \left(\norm{w_{n+1}-w_n}-\norm{z_{n+1}-z_n}\leq \frac{1}{k+1}\right).
\end{equation}
Let $N \in \N$ be such that $\forall n \in \N (\norm{w_n-z_n}\leq N)$.
Then 
\begin{equation*}
\begin{split}
\forall &k,l\in \N \,\forall t \in \N \setminus\{0\}\,\mforall f:\N \to \N \,\exists m \in [l,\varphi(k,f)]\,\exists p< R(a,k,t)\cdot N\\&\Bigg[\left(\norm{w_{m+t}-z_m}-\left(1+\sum_{i=0}^{t-1}\alpha_{m+i}\right)\frac{p+1}{R(a,k,t)}\geq -\frac{1}{k+1} \right) 
\wedge  \norm{w_{m+t}-z_{m+t}}\geq \frac{p}{R(a,k,t)}\\
& \quad \wedge \,\forall n \in [m,m+t+f(m)]\left(\norm{w_n-z_n}\leq \frac{p+1}{R(a,k,t)} \right)\Bigg],
\end{split}
\end{equation*}
where $R(a,k,t)=t(2t+1)a^t(k+1)$ and $\varphi(k,f)=\varphi(k,f,l,t,a,\nu,N):=\theta(R(a,k,t)-1,h(R(a,k,t)-1),t,N,g)$, with $g,h$ functions defined respectively by $g(m):=t+f(m)$ and $h(r):=h(a,l,\nu,r):=\max\{a,l, \nu(r)\}$. The function $\theta$ is as in Lemma~\ref{lemmarationalapprox2}.
\end{lemma}

\begin{proof}

Let $r,l \in \N$, $t \in \N \setminus \{0\}$ be arbitrary and $f:\N\to\N$ be a monotone function. Define $M:=h(r)=\max\{a,l, \nu(r)\}$. Considering $x_n:=\norm{w_n-z_n}$ and $\theta=\theta(r,M,t,N,g)$ we apply Lemma~\ref{lemmarationalapprox2} to $r,M,t$ and $g:\N \to \N$ defined by $g(m)=t+f(m)$ to find $p<N(r+1)$ and $m \in [M,\theta]$ such that 
\begin{equation}\label{appllimsup}
\norm{w_{m+t}-z_{m+t}}\geq \frac{p}{r+1} \wedge \forall n \in [m,m+g(m)]\left(\norm{w_n-z_n}\leq\frac{p+1}{r+1} \right).
\end{equation}

Furthermore, for  $n\geq m$ we have
\begin{equation}\label{formulaconj}
\alpha_n \leq 1-\frac{1}{a}\wedge \norm{w_{n+1}-w_n}-\norm{z_{n+1}-z_n}\leq \frac{1}{r+1}.
\end{equation}
From \eqref{appllimsup}, \eqref{formulaconj} and the fact that $M \geq l$ we conclude that
\begin{equation}\label{eqconjunta}
\begin{split}
\forall &r,l \in \N \,\forall t \in \N\setminus \{0\}\,\mforall f:\N \to \N \,\exists m \in [l,\theta] \,\exists p<N(r+1)\\ &\left[\forall n \geq m\left(\alpha_n \leq 1- \frac{1}{a} \, \wedge \right. \norm{w_{n+1}-w_n}-\norm{z_{n+1}-z_n}\leq \frac{1}{r+1}\right) \wedge \, \norm{w_{m+t}-z_{m+t}}\geq \frac{p}{r+1} \\
&\left.  \quad \wedge \, \forall n \in [m,m+t+f(m)]\left(\norm{w_n-z_n}\leq \frac{p+1}{r+1}\right)\right].
\end{split}
\end{equation}

Working with $m,p$ given by \eqref{eqconjunta}, we now argue that for all $j \leq t-1$.
\begin{equation}\label{ineq5}
\norm{w_{m+t}-z_{m+j}}\geq \left(1+ \sum_{i=j}^{t-1}\alpha_{m+i} \right)\frac{p+1}{r+1}- \frac{(t-j)(2t+1)a^{t-j}}{r+1}.
\end{equation}
We have
\begin{equation*}
\begin{split}
\frac{p}{r+1} &\leq \norm{w_{m+t}-z_{m+t}} \\
&= \norm{w_{m+t}-\alpha_{m+t-1}w_{m+t-1}-(1-\alpha_{m+t-1})z_{m+t-1}}  
\\ &\leq \alpha_{m+t-1}\norm{w_{m+t}-w_{m+t-1}}+(1-\alpha_{m+t-1})\norm{w_{m+t}-z_{m+t-1}} 
\\ & \leq \alpha_{m+t-1}\norm{z_{m+t}-z_{m+t-1}}+\frac{1}{r+1}+(1-\alpha_{m+t-1})\norm{w_{m+t}-z_{m+t-1}} 
\\ & = \alpha^2_{m+t-1}\norm{w_{m+t-1}-z_{m+t-1}}+\frac{1}{r+1}+(1-\alpha_{m+t-1})\norm{w_{m+t}-z_{m+t-1}} 
\\ & \leq  \alpha^2_{m+t-1}\frac{p+1}{r+1}+\frac{1}{r+1}+(1-\alpha_{m+t-1})\norm{w_{m+t}-z_{m+t-1}}. 
\end{split}
\end{equation*} 
Hence
\begin{equation*}
\begin{split}
\norm{w_{m+t}-z_{m+t-1}} &\geq \frac{\left(1- \alpha^2_{m+t-1}\right)\frac{p+1}{r+1}- \frac{2}{r+1}}{1-\alpha_{m+t-1}}\\
&=\left(1+\alpha_{m+t-1}\right)\frac{p+1}{r+1}-\frac{2}{(r+1)(1-\alpha_{m+t-1})}\\
&\geq \left(1+\alpha_{m+t-1}\right)\frac{p+1}{r+1}-\frac{2a}{r+1}\\
& \geq \left(1+\alpha_{m+t-1}\right)\frac{p+1}{r+1}-\frac{(2t+1)a}{r+1}.
\end{split}
\end{equation*}
So, $\eqref{ineq5}$ holds for $j=t-1$. To conclude we assume that $\eqref{ineq5}$ holds for some $j \in [1,t-1]$ and want to see that it holds for $j-1$. Since
\begin{equation*}
\begin{split}
\left(1+ \sum_{i=j}^{t-1} \alpha_{m+i}\right)\frac{p+1}{r+1}& - \frac{(t-j)(2t+1)a^{t-j}}{(r+1)} \leq\norm{w_{m+t}-z_{m+j}} \\
& = \norm{w_{m+t}- \alpha_{m+j-1}w_{m+j-1}-(1-\alpha_{m+j-1})z_{m+j-1}}\\
& \leq \alpha_{m+j-1}\norm{w_{m+t}-w_{m+j-1}}+(1-\alpha_{m+j-1})\norm{w_{m+t}-z_{m+j-1}}\\
& \leq \alpha_{m+j-1} \sum_{i=j-1}^{t-1}\norm{w_{m+i+1}-w_{m+i}}+(1-\alpha_{m+j-1})\norm{w_{m+t}-z_{m+j-1}}\\
& \leq \alpha_{m+j-1} \sum_{i=j-1}^{t-1}\norm{z_{m+i+1}-z_{m+i}}+\frac{t}{r+1} +(1-\alpha_{m+j-1})\norm{w_{m+t}-z_{m+j-1}}\\
& = \alpha_{m+j-1} \sum_{i=j-1}^{t-1} \alpha_{m+i}\norm{w_{m+i}-z_{m+i}}+\frac{t}{r+1} +(1-\alpha_{m+j-1})\norm{w_{m+t}-z_{m+j-1}}\\
& \leq  \alpha_{m+j-1} \sum_{i=j-1}^{t-1} \alpha_{m+i}\left(\frac{p+1}{r+1}\right)+\frac{t}{r+1} +(1-\alpha_{m+j-1})\norm{w_{m+t}-z_{m+j-1}}, \\
\end{split}
\end{equation*}
we obtain that
\begin{equation*}
\begin{split}
\norm{w_{m+t}-z_{m+j-1}} &\geq \frac{p+1}{r+1}\frac{\left(1+ \sum_{i=j}^{t-1}\alpha_{m+i} \right)-\alpha_{m+j-1}\sum_{i=j-1}^{t-1}\alpha_{m+i}}{1-\alpha_{m+j-1}}-\frac{(t-j)(2t+1)a^{t-j}+t}{(r+1)(1-\alpha_{m+j-1})}\\
&\geq \frac{p+1}{r+1}\frac{1+ \sum_{i=j}^{t-1}\alpha_{m+i}\left(1-\alpha_{m+j-1} \right)-\alpha^2_{m+j-1}}{1-\alpha_{m+j-1}}-\frac{(t-j)(2t+1)a^{t-j+1}+ta}{r+1}\\
&\geq \frac{p+1}{r+1} \left(\frac{1-\alpha^2_{m+j-1}}{1-\alpha_{m+j-1}}+\sum_{i=j}^{t-1}\alpha_{m+i} \right)-\frac{(t-j)(2t+1)a^{t-j+1}+(2t+1)a^{t-j+1}}{r+1}\\
&= \left(1+\sum_{i=j-1}^{t-1}\alpha_{m+i} \right)\frac{p+1}{r+1}-\frac{(t-j+1)(2t+1)a^{t-j+1}}{r+1},
\end{split}
\end{equation*}
which implies that $\eqref{ineq5}$ holds for all $j \leq t-1$ as we wanted.
Instantiating $j$ with $0$ in $\eqref{ineq5}$ we obtain
\begin{equation*}
\norm{w_{m+t}-z_m}\geq \left(1+\sum_{i=0}^{t-1}\alpha_{m+i} \right)\frac{p+1}{r+1}- \frac{t(2t+1)a^{t}}{r+1}.
\end{equation*}
Hence
\begin{equation}\label{formulasummary}
\begin{split}
\forall &r,l \in \N \,\forall t \in \N \setminus \{0\}\,\mforall f:\N \to \N \,\exists m \in[l,\theta(r,h(r),t,N,g)] \,\exists p<N(r+1)\\
&\Bigg[\norm{w_{m+t}-z_m}\geq \left(1+\sum_{i=0}^{t-1}\alpha_{m+i} \right)\frac{p+1}{r+1}- \frac{t(2t+1)a^{t}}{r+1}\, \wedge \norm{w_{m+t}-z_{m+t}}\geq \frac{p}{r+1}\\
& \quad \wedge \, \forall n \in [m,m+t+f(m)]\left(\norm{w_n -z_n}\leq \frac{p+1}{r+1} \right)\Bigg].
\end{split}
\end{equation}
Given $k \in \N$, we conclude the result by putting $r=R(a,k,t)-1$ in \eqref{formulasummary}. 
\end{proof}

 \begin{lemma}\label{LemmaSuzuki2}
Let $(z_n),(w_n)$ be sequences in a normed space $X$ and $N \in \N$ be such that $\norm{z_n},\norm{w_n}\leq N$, for all $n \in \N$. Let $(\alpha_n) \subset [0,1]$ be a sequence of real numbers and $a \in \N \setminus \{0\}$ be such that $\forall n \geq a \left(\frac{1}{a}\leq\alpha_n \leq 1 - \frac{1}{a}\right).$
Suppose that $z_{n+1}=\alpha_nw_n+(1-\alpha_n)z_n$, for all $n \in \N$ and that there exists a monotone function $\nu: \N \to \N$ such that 
\begin{equation}\label{eqnu}
\forall k \in \N \,\forall n \geq \nu(k) \left(\norm{w_{n+1}-w_n}-\norm{z_{n+1}-z_n}\leq \frac{1}{k+1}\right).
\end{equation}
Then 
\begin{equation*}
\forall k\in \N \,\mforall f: \N \to \N \,\exists n \leq \widetilde{\chi}(k,f)\, \forall m \in [n,n+f(n)] \left(\norm{w_m - z_m} \leq \frac{1}{k+1} \right),
\end{equation*}
where $\widetilde{\chi}(k,f)=\widetilde{\chi}(k,f,a,\nu,N)=\varphi(k,f,a,t,a,\nu,2N)$, with $\varphi$ is as in Lemma \ref{lemmaSuzuki1} and $t:=\max\{2Na(k+1),1\}$.
\end{lemma}

\begin{proof}
Suppose towards a contradiction that there exist $k_0 \in \N$ and a monotone function $f_0$ such that   
\begin{equation*}
\forall m \leq \widetilde{\chi}(k_0,f_0) \,\exists n \in [m,m+f_0 (m)]\left( \norm{w_n-z_n}>\frac{1}{k_0+1}\right).
\end{equation*}
Define $t_0:=\max\{2Na(k_0+1),1\}$. We have that $t_0 \geq 1$ and $\left(1+\dfrac{t_0}{a}\right)\dfrac{1}{k_0+1} \geq 2N+ \dfrac{1}{k_0+1}$.
Applying Lemma \ref{lemmaSuzuki1} with $k=k_0$, $l=a$, $t=t_0$ and $f=f_0$, we find $p\in \N$ and $m \in [a, \varphi(k_0,f_0,a,t_0,a,\nu,2N)]=[a, \widetilde{\chi}(k_0,f_0)]$ such that 
\begin{equation*}
\begin{split}
\norm{w_{m+t_0}-z_m}&\geq \left(1+ \sum_{i=0}^{t_0-1}\alpha_{m+i} \right)\frac{p+1}{R(a,k_0,t_0)}-\frac{1}{k_0+1}\wedge \norm{w_{m+t_0}-z_{m+t_0}}\geq \frac{p}{R(a,k_0,t_0)}\\
&  \quad \wedge \,\forall n \in [m, m+t_0+f_0(m)]\left(\norm{w_n-z_n}\leq\frac{p+1}{R(a,k_0,t_0)}\right).
\end{split}
\end{equation*}
We have $[m,m+f_0(m)]\subseteq[m,m+t_0+f_0(m)]$, and then $$\frac{1}{k_0+1}<\norm{w_m-z_m}\leq\frac{p+1}{R(a,k_0,t_0)}.$$   Hence
\begin{equation*}
\begin{split}
2N+ \frac{1}{k_0+1}&\leq\left(1+\frac{t_0}{a}\right)\frac{1}{k_0+1}\\
&\leq \left(1+\sum_{i=0}^{t_0-1}\alpha_{m+i} \right)\frac{1}{k_0+1}\\
&< \left(1+\sum_{i=0}^{t_0-1}\alpha_{m+i} \right)\frac{p+1}{R(a,k_0,t_0)}\\
&\leq \norm{w_{m+t_0}-z_m}+\frac{1}{k_0+1}\leq 2N+\frac{1}{k_0+1},
\end{split}
\end{equation*}
a contradiction.
\end{proof}

\subsection{Metastability revisited}\label{Sectionrevisited}

\begin{remark}\label{witQS}
Under the conditions of  Lemma~\ref{lemmamain1}, the function $\nu(k):=\max\{\Gamma(10acN_0(k+1)),\ell(10a(N_0+N_1+N_3)(k+1)),E(5a(k+1))+1\}$ satisfies \eqref{eqnu}. Furthermore, $\norm{z_n},\norm{w_n} \leq 2aN_0+N_1+N_3$. Hence, Lemma~\ref{LemmaSuzuki2} with $N=2aN_0+N_1+N_3$, $\alpha_n=1-\gamma_n$  is satisfied with the function $\nu$ and outputs the function $\chi_0$ defined by $$\chi_0(k,f):=\widetilde{\chi}(k,f,a,\nu,2aN_0+N_1+N_3),$$ which satisfies \eqref{hypchi}.  Note that having this function $\chi_0$, the hypothesis \eqref{hypchi} can be removed from Proposition~\ref{lemmaintermediate}. Moreover, the bound in Proposition~\ref{lemmaintermediate}\eqref{limzn-J} can be simplified to $\chi_0(2a(k+1),\tilde{f}_k)$, and similarly for $\xi_{\chi_0}$.
\end{remark}

\begin{notation}
We write $\xi,\psi$ and $\Psi$, instead of $\xi_{\chi_0},\psi_{\chi_0}$ and $\Psi_{\chi_0}$, respectively, where $\chi_0$ is the functional defined in Remark~\ref{witQS}.
\end{notation}

Having an explicit function satisfying \eqref{hypchi} we can now prove metastability for \eqref{PPA} with the initial quantitative assumptions  $(Q_{\ref{hyp1}})-(Q_{\ref{ineqerror1}})$.

\begin{theorem}\label{theoremyaonoor2}
Let $(z_n)$ be generated by \eqref{PPA}. Assume that there exist $a,c \in \N \setminus\{0\}$ and monotone functions $\ell,L,\Gamma,E$ such that $(Q_{\ref{hyp1}})-(Q_{\ref{ineqerror1}})$ hold. Let $\mathcal{C}:\N\to\N$ be a monotone function such that $c_n \leq\mathcal{C}(n)$, for all $n \in \N$. Let $N_1, N_2, N_3 \in \N$ be such that $N_1 \geq \norm{u}$, $N_2 \geq \sum_{i=0}^{E(0)}\norm{e_i}+1$, and for some $s \in S$ one has $N_3 \geq \max\{\norm{u-s},\norm{z_0-s}\}$. Define $N:=\max\{2N_3,N_2+N_3\}$.
Let $(w_n)$ be such that $z_{n+1}=\gamma_nz_n+(1-\gamma_n)w_n$.
Then
\begin{equation*}
\forall k \in \N \,\forall f : \N \to \N \,\exists n \leq \phi(k,f)\, \forall i,j \in [n,n+f(n)] \left(\norm{z_i-z_j}\leq \frac{1}{k+1} \right),
\end{equation*}
where $\phi(k,f):=\phi_{\chi_0}(k,f^{\mathrm{maj}})$, with $\phi_{\chi_0}(k,f):=\Theta(4(k+1)^2-1,\lambda m\ldotp(m+f(m)),L,\Psi, G,4N^2)$, $G(k)=E(M_2(k+1))$, $\Theta$ as in Lemma~\ref{lemmaqtXu2}, $\Psi$ given by Lemma~\ref{psi}, $M_1:=3 N_2+4N$ and $M_2=M_1+2(N_3+N)$.
\end{theorem}

As an application we consider the special case where the sequence $(c_n)$ is constant. We note that this case was also considered in Yao and Noor's paper \cite[Corollary~3.1]{YaoNoor2008}.

Let $c_0$ be a positive real number. Consider the sequence $(c_n)$ constantly equal to $c_0$, and $(z_n)$ generated by \eqref{PPA}. Assume that there exist $a,c,C \in \N\setminus \{0\}$ and monotone functions $\ell,L,E$ such that $(Q_{\ref{hyp1}})-(Q_{\ref{ineqgamma}})$ and $(Q_{\ref{ineqerror1}})$ hold and that $C \geq c_0 \geq \frac{1}{c}$. We consider $\mathcal{C}:\N \to \N$ to be the function identicaly equal to $C$. Note that we can take either $c=1$ or $C=1$.
Clearly $(Q_{\ref{ineqcn}})$, and $(Q_{\ref{Q5}})$ with $\Gamma(k)\equiv 0$ are satisfied. The definition of $\nu$ in Lemma~\ref{lemmamain1} simplifies to $$\nu(k)=\max\{\ell(8 a(N_0+N_1+N_3)(k+1)),E(4a(k+1))+1\},$$ 
which causes changes in $\chi_0$ and consequently in $\xi$, $\psi$ and $\Psi$. This simplifies the bound $\phi$ obtained in Theorem~\ref{theoremyaonoor2}.

\section{Final considerations}\label{Sectionlogic}

In Theorem~\ref{theoremyaonoor2} we obtain a bound on the metastability of \eqref{PPA} which is uniform on the parameters of the algorithm. Let us elaborate on this. The bound is uniform on the choice of the anchor point $u$, the given initial point $z_0$ and a point $s \in S$ witnessing the assumption that $S$ is nonempty, depending only on natural numbers $N_1$ and $N_3$. The dependence on the sequences $(\lambda_n), (\gamma_n)$ and $(c_n)$ is  respectively only in the form of a rate of convergence $\ell$ satisfying $(Q_{\ref{hyp1}})$ and a rate of divergence $L$ satisfying $(Q_{\ref{ratediv}})$; a natural number $a$ satisfying $(Q_{\ref{ineqgamma}})$; and a natural number $c$ satisfying $(Q_{\ref{ineqcn}})$, a rate of convergence $\Gamma$, satisfying $(Q_{\ref{Q5}})$ and on a majorizing function  $\mathcal{C}$. Finally, it is uniform on the error sequence $(e_n)$, as it depends only on a Cauchy rate $E$ satisfying $(Q_{\ref{ineqerror1}})$ and on a natural number $N_2$.

It is well-know that for $(\lambda_n)\subset (0,1)$ the condition $\sum \lambda_n=\infty$ is equivalent to the condition $\prod (1-\lambda_n)=0$. Although we don't do this here, one could have worked with a rate of convergence towards zero $L'$ for $\left(\prod (1-\lambda_n)\right)$ instead of the rate of divergence $L$ for $\left(\sum \lambda_n\right)$. This is done for similar quantitative analyses in \cite{LLPP(ta),PP(ta)}. In certain cases, that option may prove useful as a function $L'$ may be of lower complexity than that of a function $L$, e.g.\ the sequence $\lambda_n=\frac{1}{n+1}$ has a liner rate $L'$ but an exponential rate $L$.

Metastability and the Cauchy property are equivalent. So, Theorem~\ref{theoremyaonoor2} and the fact that the space $ H$ is complete imply that \eqref{PPA} is strongly convergent. Similarly, from Proposition~\ref{lemmaintermediate} (with instantiation $\chi=\chi_0$) we conclude that $\lim \norm{\mathsf{J}(z_n)-z_n}=0$ and thus \eqref{PPA} converges to a zero of the maximal monotone operator $T$, say $\tilde{z}$. To see that $\tilde{z}$ must be the projection point $P_{S}(u)$ first note that since $\tilde{z} \in S$ we have that $\langle u- P_{S}(u),\tilde{z} -P_{S}(u)\rangle \leq 0$. Since $z_n \to \tilde{z}$, the conclusion of Lemma~\ref{psi} is satisfied with $p=P_{S}(u)$. Finally, following the arguments in the proof of Theorem~\ref{theoremyaonoor} with $p=P_{S}(u)$, we see that \eqref{finaleq} holds with the point $p$ always equal to $P_{S}(u)$, implying that $\tilde{z}=P_{S}(u)$. This shows that the analysis in this paper indeed corresponds to a quantitative analysis of the original proof by Yao and Noor.

We finish with some considerations concerning the logical aspects of our analysis.

Let us start by pointing out that, in principle, the monotone functional interpretation could be used to analyse the results presented in this paper. However, our elimination of the sequential weak compactness argument can be seen as an application of the general method obtained in \cite{FFLLPP(19)}. Also, using the $\BFI$ enables us to make use of Proposition~\ref{projection} and Lemma~\ref{lemmaqtXu1} (shown in \cite{PP(ta)} using the $\BFI$) which makes our analysis easier to carry out. 

As already mentioned, the original proof of Theorem~\ref{ThmYaoNoor} requires strong principles. These principles are countable choice as used in the projection argument, sequential weak compactness  and  the existence of the $\limsup$ in Lemma~\ref{SuzukiLemma1}, which require arithmetical comprehension. The analysis of these principles require the use of a stronger form of recursion called \emph{bar recursion} \cite{S(62),O(06)}. As shown below, the arguments in Section~\ref{SectionRestricting} allow us to avoid the use of bar recursion. 
There are already many examples in the proof mining literature  where it is possible to avoid the use of arithmetical comprehension\footnote{We would like to thank Ulrich Kohlenbach for pointing this out to the first author and for providing the appropriate references.}. For example, in \cite{K(05),KLN(18)} the use of the existence of a limit point for a sequence in a compact geodesic space (which requires arithmetical comprehension) is replaced by a combinatorial argument. In \cite{KLAN(ta),KLAN(17)}, a proof of an asymptotic regularity theorem that was based on countable nested uses of sequential compactness (and hence arithmetic comprehension) is analysed resulting in a simple exponential bound, by elimination of sequential compactness. Moreover, in a series of papers, Kohlenbach showed how the monotone functional interpretation can be used to replace the use of arithmetical comprehension by optimal arithmetic substitutes (see e.g. \cite{K(98),K(00)} and \cite[Section~17.9]{K(08)}).

The elimination of countable choice required for the projection argument is carried out in Section~\ref{SectionProjection}. 
The key observation is that \eqref{projarg} can be replaced by \eqref{projarg_weak}. This is in line with earlier analyses (see for example \cite{kohlenbach2011quantitative,FFLLPP(19),PP(ta)}) and it is well-known that this allows for the extracted quantitative information to be expressed in terms of G\"odel's primitive recursive functionals.

The way to deal with sequential weak compactness is explained in full detail in \cite{FFLLPP(19)}, in the context of the $\BFI$. The key point is that sequential weak compactness can be replaced by countable Heine/Borel compactness. The content of Section~\ref{Sectionswc} shows that it can be adapted to our context in a similar way. 

In Section~\ref{SectionBypassACA} we made the modifications necessary to bypass the assumed existence of the real number $d = \limsup \norm{w_n -z_n}$ in Lemma~\ref{SuzukiLemma1}.  
In \cite{KS(ta)}, Kohlenbach and Andrei Sipo\c{s} give a rational approximation to the $\limsup$ of a certain sequence by interpreting the approximation. As described in detail below, we must deal with a similar issue. 

Let $N \in \N$ and let $(x_n)$ be a sequence of real numbers contained in the interval $[0,N]$. The existence of the $\limsup x_n$ can be stated as
\begin{equation*}\label{limsup1}
\exists d\in\R \,\forall k\in\N\, \left( \forall n\in \N \, \exists m\geq n \, \left(x_m\geq d-\frac{1}{k+1}\right) \land \exists n'\in \N \,\forall m'\geq n' \,\left(x_{m'}\leq d+\frac{1}{k+1}\right)\right).
\end{equation*}
The main point is that we can weaken this statement by switching the outermost quantifiers
\begin{equation}\label{limsup2}
\forall k\in \N\,\exists d\in\R \, \left( \forall n\in \N \, \exists m\geq n \, \left(x_m\geq d-\frac{1}{k+1}\right) \land \exists n'\in \N \,\forall m'\geq n' \,\left(x_{m'}\leq d+\frac{1}{k+1}\right)\right).
\end{equation}
In fact, we will show that such $d$ in \eqref{limsup2} is already witnessed by a rational number satisfying
\begin{equation}\label{limsup3}
\forall k\in\N \,\exists p< N(k+1)\, \left( \forall n\in \N \,\exists m\geq n \, \left(x_m\geq \frac{p}{k+1}\right) \land \exists n'\in \N \,\forall m'\geq n' \,\left(x_{m'}\leq \frac{p+1}{k+1}\right)\right),
\end{equation}
which implies that \eqref{limsup2} holds with $d=\frac{p}{k+1}$.

The idea behind \eqref{limsup3} is the following. For each $k\in\N$, by dividing the interval $[0,N]$ into subintervals of length $\frac{1}{k+1}$, there exists $p<N(k+1)$ such that $\frac{p}{k+1}\leq \limsup x_n \leq \frac{p+1}{k+1}$. If we take $d=\frac{2p+1}{2(k+1)}$, i.e.\ the middle point, then it should satisfy \eqref{limsup2} for $2k+1$ (and hence for $k$). This results in statement \eqref{limsup3}.

Lemma~\ref{lemmaratap}, which is shown by $\Pi^0_1$-induction, can be seen to imply \eqref{limsup3} using a collection argument. First note that Lemma~\ref{lemmaratap} implies
\begin{equation}\label{limsup4}
\begin{split}
\forall k,n \in \N\, \mforall f:\N \to \N \, &\exists p<N(k+1)\\ 
&\left(\exists m\geq n \left(x_m \geq \frac{p}{k+1}\right) \wedge \exists n' \in \N \, \forall m' \in [n',n'+f(n')] \left(x_{m'} \leq \frac{p+1}{k+1}\right)\right).
\end{split}
\end{equation}

By a collection argument, we conclude
\begin{equation}\label{limsup5}
\begin{split}
\forall k \in \N \,\exists p<N(k+1)\, &\forall n\in \N\, \mforall f:\N \to \N\\ 
&\left(\exists m\geq n \left(x_m \geq \frac{p}{k+1}\right) \wedge \exists n' \in \N \, \forall m' \in [n',n'+f(n')] \left(x_{m'} \leq \frac{p+1}{k+1}\right)\right),
\end{split}
\end{equation}
which by (monotone) choice axiom is equivalent to \eqref{limsup3}.

Let us elaborate on \eqref{limsup5}. Clearly it is a sufficient condition to \eqref{limsup4}. Let see that \eqref{limsup4} implies \eqref{limsup5}. Assuming that \eqref{limsup5} fails, we obtain that for some $k \in \N$ 
\[\forall r\leq N(k+1)\, \exists n \, \mexists f: \N \to \N \,\forall p< r \left(\forall m\geq n \left(x_m < \frac{p}{k+1}\right) \lor \forall n' \in \N \, \exists m' \in [n',n'+f(n')] \left(x_{m'} > \frac{p+1}{k+1}\right)\right).
\]
By instantiating $r=N(k+1)$, one concludes that \eqref{limsup4} must also fail.

Of course, this collection argument is fully justified by a form of induction. The reader can compare this way of proving \eqref{limsup3}, using $\Pi_1^0$-induction and a collection argument, to the similar \cite[Proposition 4.2]{KS(ta)}, where $\Pi^0_2$-induction was used.

\section*{Acknowledgements}

We would like to thank Lauren\c{t}iu Leu\c{s}tean for suggesting us to analyse Yao and Noor's theorem and for several discussions concerning the subject of this paper. Our paper also benefits from discussions and remarks by Fernando Ferreira and Ulrich Kohlenbach.

Both authors acknowledge the support of FCT - Funda\c{c}\~ao para a Ci\^{e}ncia e Tecnologia under the project: UID/MAT/04561/2019 and the research center Centro de Matem\'{a}tica, Aplica\c{c}\~{o}es Fundamentais e Investiga\c{c}\~{a}o Operacional, Universidade de Lisboa. 
The second author also acknowledges the support of the `Future Talents' short-term scholarship at Technische Universit{\"a}t Darmstadt.

\bibliography{References}{}
\bibliographystyle{plain}

\end{document}